\documentclass[
final
]{dmtcs-episciences}

\usepackage{amsmath,amsthm}

\usepackage[utf8]{inputenc}
\usepackage{subfigure}
\usepackage{xcolor}

\usepackage[round]{natbib}
\usepackage{enumerate}
\usepackage[shortlabels]{enumitem}

\newcommand{\MyQuote}[1]{\vspace{0cm}
	\parbox{15cm}{\em #1}\hspace*{2cm}\\[0cm]}

\newenvironment{mylist}{
\begin{description}[style=multiline, leftmargin=1.7cm,font=\normalfont]
	  \setlength{\itemsep}{2pt}
	\setlength{\parskip}{1pt}
}
{\end{description} }

\newenvironment{property}
{\noindent\begin{description}[style=multiline, leftmargin=0.085\linewidth,font=\normalfont]
	}
	{\end{description}}

\def\PP{{\mathcal P}}
\def\QQ{{\mathcal Q}}

\def\RR{{\mathcal R}}

\def\MT{{\mathbb T}}
\def\ML{{\mathbb L}}
\def\MP{{\mathbb P}}

\def\halnospace{\hfill $\triangle$}

\newenvironment{prf}[1]{\noindent\textbf{Proof of Claim~{#1}:}}{\hfill\halnospace}

\newtheorem{theo}{Theorem}
\newtheorem{theorem}{Theorem}[section]
\newtheorem{lemma}[theorem]{Lemma}
\newtheorem{conj}[theorem]{Conjecture}
\newtheorem{prop}[theorem]{Proposition}
\newtheorem{obs}[theorem]{Observation}
\newtheorem{claim}[theorem]{Claim}

\title{On the Connectivity of Token Graphs of Trees}
\author{
	Ruy Fabila-Monroy\affiliationmark{1}\thanks{Partially supported by CONACYT (Mexico), grant 253261.}
	\and
	Jes\'us Lea\~nos\affiliationmark{2}
	\and 
	Ana Laura Trujillo-Negrete\affiliationmark{1}\thanks{Partially supported by CONACYT (Mexico), Convocatoria 2021 de Estancias Posdoctorales por M\'exico en Apoyo por SARS-CoV-2 (COVID-19).}
}

\affiliation{Departamento de Matem\'aticas, Cinvestav, CDMX, Mexico.\\
	Unidad Acad\'emica de Matem\'aticas, Universidad Aut\'onoma de Zacatecas, Zacatecas, Mexico.}
\keywords{token graphs, connectivity, trees}
\received{2021-06-03}

\revised{2021-11-06}

\accepted{2022-02-14}

\begin{document}

\publicationdetails{24}{2022}{1}{9}{7538}
\maketitle

\begin{abstract}
  Let $k$ and $n$ be integers such that  $1\leq k \leq n-1$, and let $G$ be a simple graph of order $n$.
  The $k$--token graph $F_k(G)$ of $G$ is the graph whose vertices are the $k$-subsets of $V(G)$, where two vertices 
  are adjacent in $F_k(G)$ whenever their symmetric difference is an edge of $G$. 
  In this paper we show that if $G$ is a tree, then the connectivity of $F_k(G)$ is equal to the minimum degree of 
  $F_k(G)$.
\end{abstract}

\section{Introduction}

Throughout this paper, $G$ is a simple finite graph of order $n\geq 2$ and
$k\in \{1, \ldots , n-1\}$. The \emph{$k$-token graph} $F_k(G)$ of $G$ is the graph whose vertices are all the $k$-subsets of 
$V(G)$, where two $k$-subsets are adjacent whenever their symmetric difference is a pair of adjacent vertices in $G$.
We often write token graph instead of $k$-token graph. See Figure~\ref{figs:example_tokengraph} for an example. 

\begin{figure}[t]
	\begin{center}
	\includegraphics[width=0.6\textwidth]{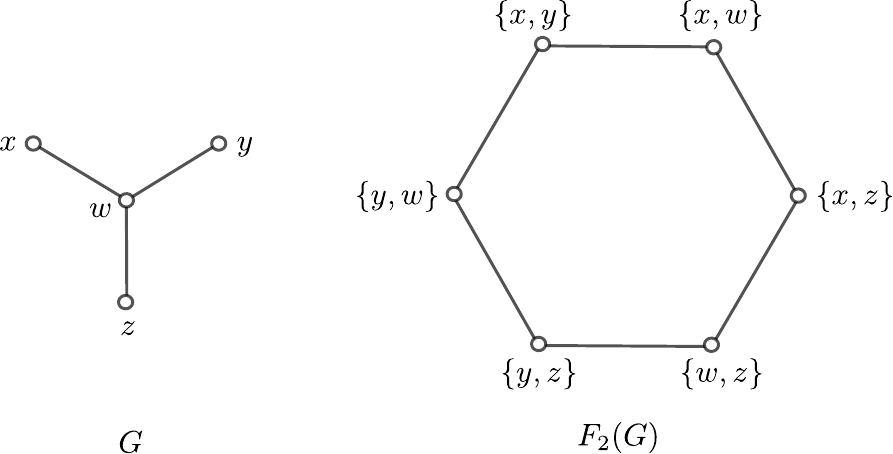}
	\caption{A graph $G$ and its $2$--token graph $F_2(G)$.}
	\label{figs:example_tokengraph}
\end{center}
\end{figure}

The study of  token graphs probably started with 
\cite{distance} PhD Thesis,
in which $F_k(G)$ was called
the  \emph{$k$-subgraph graph} of $G$ and some results concerning the diameter of $F_k(G)$ were reported. Since then, 
token graphs have been defined independently at least three more times.

\cite{double_1} defined $F_2(G)$ and call it the \emph{double vertex graph} of $G$, and a year later, \cite{alavi2} 
generalized the concept for $k\in \{2,\ldots ,n-1\}$ under the name of \emph{$k$-tuple vertex graph} of $G$. 
In these two papers, the authors studied several combinatorial issues of $F_k(G)$ such as Eulerianicity, Hamiltonicity, connectivity, regularity, etc. 	

This concept was reintroduced for third time by \cite{rudolph}, when some connections of $F_k(G)$ with quantum mechanics
and the graph isomorphism problem were discussed. Regarding the quantum mechanics, Rudolph used $F_k(G)$ to model the evolution
of a cluster of $n$ interacting qubits ($2$-level atoms), which must have exactly $k$ qubits in excited state at any time 
(the $n$ qubits are the vertices of $G$ and their interactions define the edges of $G$). The use of $F_k(G)$ in this 
direction is still of interest, see~\cite{barghi,alzaga,yingkai}. For instance, ~\cite{yingkai} showed that 
$F_k(G)$ has applications in the Heisenberg model, which is a quantum theory of magnetism.

With respect to the graph isomorphism problem, \cite{rudolph} found pairs of cospectral graphs $G$ and $H$ such that  
$F_2(G)$ and  $F_2(H)$ are not cospectral, implying that $G$ and $H$ are not isomorphic. Following Rudolph's 
approach, 
\cite{rudolph2} showed the existence of pairs of non-isomorphic 
cospectral graphs whose corresponding $2$-token graphs are cospectral. 
A few years later, 
\cite{barghi}, 
and independently, 
\cite{alzaga}, showed that for any $k\in \mathbb{Z}^+$, there exist 
infinitely many pairs of non-isomorphic graphs whose corresponding $k$-token graphs are cospectral. In 
\cite{rudolph}
$F_k(G)$ was originally called the \emph{$k$-level matrix} of $G$, but in
\cite{rudolph2} $F_k(G)$ was renamed as 
the \emph{symmetric $k$-th power} of $G$.

As far as we know,~\cite{fabila2012tokengraphs} is the last paper in which $F_k(G)$ has been defined, under the 
name of the \emph{$k$-token graph} of $G$. In that paper, Fabila-Monroy, Flores-Pe\~naloza, Huemer, Hurtado, Urrutia, 
and Wood defined $F_k(G)$ as  ``a model in which $k$ indistinguishable tokens 
move from vertex to vertex along the edges of a graph" and showed several results on the connectivity, diameter, 
chromatic and clique numbers, and Hamiltonian paths. From this last definition of $F_k(G)$, it is not hard to see 
that the estimation of any parameter involving connectivity or the determination of the distance between vertices 
of $F_k(G)$ can be seen as a reconfiguration problem. We recall that reconfiguration problems are a family of 
combinatorial problems that ask if there exists a step-by-step transformation between two feasible solutions of 
a problem such that all intermediate results are also feasible. For two specific examples of theses connections 
we refer the reader to~\cite{demaine,yama}.  

In 2017 Sloane\footnote{https://oeis.org/A085680} observed that the problem of determining the maximum size of 
a binary code of length $n$ and constant weight $2$ that can correct a single adjacent transposition is equivalent 
to determining the packing number of $F_2(P_{n})$, where $P_n$ is the path graph of order $n$. 
\cite{gomez} 
solved this problem.


Token graphs are also a generalization of Johnson graphs: if $G$ is the complete graph of order $n$, then $F_k(G)$ is isomorphic to the
Johnson graph $J(n,k)$. Johnson graphs have been widely studied; the analysis of many of its 
combinatorial properties is an active area of research (see for instance~\cite{alavi,brouwer,riyono,etzion,terwilliger}).


The following approach has been applied in several papers such as~\cite{fabila2012tokengraphs,alba,gomez,token2,leanos2018connectivity,leanos2019edgeconnectivity}.

\MyQuote{
	For a given graph invariant $\eta$, what can be said of $\eta(F_k(G))$ in terms of $G$ and $\eta(G)$?}
In particular, \cite{fabila2012tokengraphs} gave families of graphs of order $n$ with connectivity exactly $t$, and whose
$k$-token graphs have connectivity exactly $k(t-k+1)$, whenever $k \leq t$;  they also conjectured that 
if $G$ is $t$-connected and $k \leq t$, then $F_k(G)$ is at least  $k(t-k+1)$-connected. This was proven by \cite{leanos2018connectivity}.
Recently, a similar lower bound was proven for edge-connectivity by \cite{leanos2019edgeconnectivity}; they showed 
that if $G$ is $t$-edge-connected and $k\leq t$ then $F_k(G)$ is at least $k(t - k + 1)$-edge-connected. 
Infinite families of graphs  attaining this lower bound were also given.

In this paper we study the connectivity and edge-connectivity of $F_k(G)$  when $G$ is a tree.
As usual let $\kappa(G)$, $\lambda(G)$, and $\delta(G)$ be the \emph{connectivity, edge-connectivity,} and  
\emph{minimum degree} of $G$, respectively. 
It is well known that if $G$ is connected then
\begin{equation} \label{ineq1}
	\kappa(G)\leq \lambda(G)\leq \delta(G).
\end{equation}
The main result of this paper is the following. 
\begin{theo}\label{thm:main}
	If $G$ is a tree of order $n$ and $1\leq k\leq n-1$ then 
	\begin{displaymath}
		\kappa(F_k(G))= \lambda(F_k(G)) = \delta(F_k(G)).
	\end{displaymath}
\end{theo}

We remark that while the hypothesis $k\leq \kappa(G)$ has played a central role in both results on 
$\kappa(F_k(G))$ stated in~\cite{fabila2012tokengraphs,leanos2018connectivity}, this hypothesis does not hold when $G$ is a tree;
this absence is responsible for the new difficulties in the proof of Theorem~\ref{thm:main}. 

We now recall some standard notation which is used throughout this paper. Let $u$ and $v$ be distinct vertices of $G$. 
The distance between $u$ and $v$ in $G$ is denoted by $d_G(u,v)$ (we sometimes write $d(u,v)$ when $G$ is understood from the context); 
we write $uv$ to mean that $u$ and $v$ are adjacent.
The \emph{neighbourhood} of $v$ in $G$ is the set $\{u\in V(G): uv\in E(G)\}$ and  it is  denoted by $N_G(v)$. 
The \emph{degree} of $v$ is the number $\deg_G(v):=|N_G(v)|$. 
The number  $\delta(G):=\min\{deg_G(v): v\in V(G)\}$ is the \emph{minimum degree} of $G$.
A \emph{$u-v$ path of} $G$ is  starting at $u$ and ending in $v$. 
Let $U$ and $W$ be subsets of $V(G)$. We use: $G\setminus W$ to denote the subgraph of $G$ that results by removing $W$ from $G$;
$U\setminus W$ to denote set subtraction; and $U\triangle W$ to denote symmetric difference. For brevity, if $m$ is a positive integer, 
then we use $[m]$ to denote 
$\{1,\ldots,m\}$. We follow the convention that $[m]=\emptyset$ for $m=0$. 

The rest of the paper is organized as follows. In Section~\ref{sec:constructions} we establish several ways 
to construct paths in $F_k(G)$ which come from the concatenation of certain paths of $G$. 
These paths of $F_k(G)$ play a central role in our constructive proof of Theorem~\ref{thm:main}. 
In Section~\ref{sec:basicfacts} we give some basic results on the connectivity structure of $F_k(G)$ which 
help us to simplify significantly the proof of Theorem~\ref{thm:main}. Finally, in Section~\ref{sec:proof} 
we prove Theorem~\ref{thm:main}.


\subsection{Constructing Paths of $F_k(G)$ from Paths of $G$}\label{sec:constructions}

In this section we construct paths in $F_k(G)$ using a given set of paths of $G$.
For this purpose, we find it useful to use the following interpretation of $F_k(G)$ 
given by~\cite{fabila2012tokengraphs}. We consider that there are
$k$ indistinguishable tokens placed at the vertices of $G$ (at most one token per vertex).
A vertex of $F_k(G)$ corresponds to one of this token configurations. Two such configurations are adjacent in 
$F_k(G)$ if and only if one configuration can be reached from the other by moving one token along an edge of $G$ 
from its current vertex to an unoccupied vertex. These token moves are called \emph{admissible moves}. 
Under this interpretation, if $A$ and $B$ are two distinct $k$-subsets of $V(G)$ then a path in $F_k(G)$ 
with endvertices $A$ and $B$ corresponds to a finite sequence of token configurations that are produced by 
a corresponding sequence of admissible moves. With this in mind, now we explain how to produce some paths 
of $F_k(G)$ from a certain set of paths of $G$.

Let $P:= a_0a_1a_2\ldots a_m$ be an $a-b$ path of $G$ ($a_0=a$ and $a_m=b$); let $A,B\in V(F_k(G))$ be such that 
$A\triangle B=\{a,b\}$, $P\cap A=\{a\}$ and $P\cap B=\{b\}$.
A natural way of constructing an $A-B$ path $\PP$ in $F_k(G)$ using $P$ is by moving the token at $a$ along $P$ to $b$.
More precisely, we start at $A$, then for each $i=0,1,\ldots ,m-1$, we move  (in this order)
the token at $a_i$ along the edge $a_ia_{i+1}$ to the vertex $a_{i+1}$. We denote this sequence of admissible token moves by
\begin{displaymath}
	a_0\longrightarrow a_1\longrightarrow a_2\cdots \longrightarrow a_m.
\end{displaymath}

Clearly, the first and last configurations of this sequence correspond to the vertices $A$ and $B$ of $F_k(G)$, respectively.
Moreover, note that if $A_0=A, A_m=B$, and $A_i=(A_{i-1}\setminus \{a_{i-1}\})\cup \{a_i\}$ for $i\in [m]$, then $\PP=AA_1A_2\ldots A_{m-1}B$.
We refer to $\PP$ as the \emph{path of $F_k(G)$ induced by} $P$. See Figure \ref{figs:type1_case1}.
Let $\QQ$ be a path of $F_k(G)$ and let $\{Q_0,Q_1,\ldots ,Q_m\}$ be its vertex set. Since each of these $Q_i$'s is a $k$-set 
of $V(G)$, then $q:=k-|\cap_{i=0}^mQ_i|$ is well defined. We  say that $\QQ$ is a path \emph{of Type $q$}. Thus, $\PP$
and any edge of $F_k(G)$ are examples of paths of Type 1. 

\begin{figure}[h!]
	\begin{center}
	\includegraphics[width=0.7\textwidth]{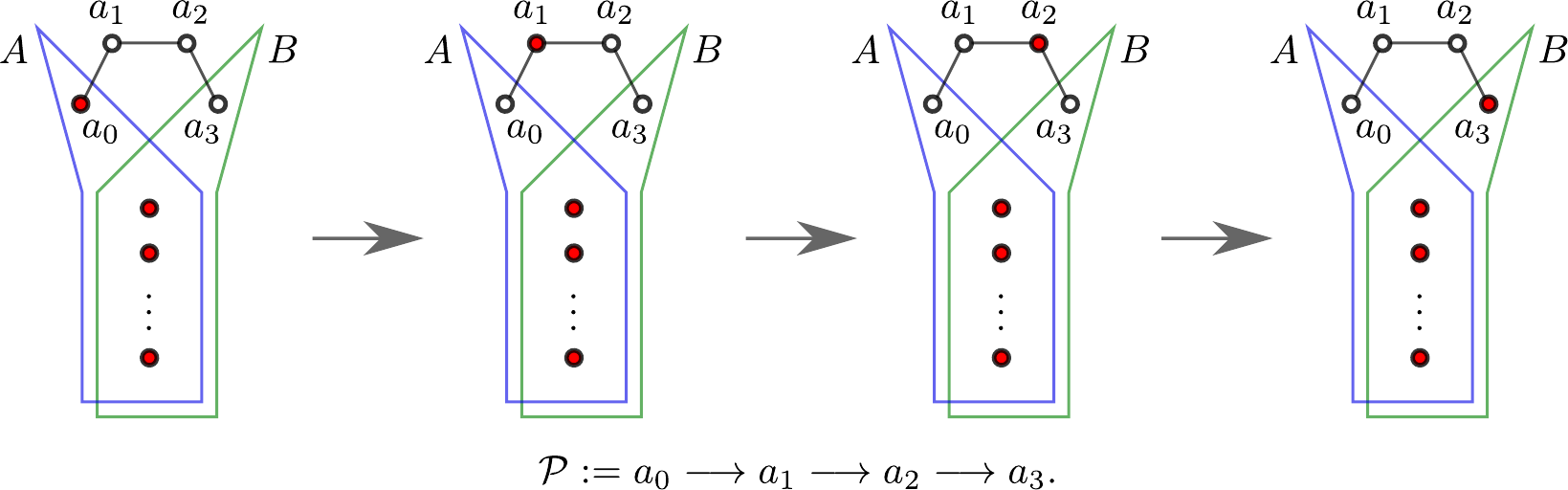}
	\caption{Four configurations of $G$. The set of red vertices of $G$ defining the left (respectively, right) configuration
		corresponds to the vertex $A$  (respectively, $B$) of $F_k(G)$. These four configurations together 
		(from left to right) define an $A-B$ path $\PP$ of $F_k(G)$. The path $\PP$ is induced by $P=a_0a_1a_2a_3$, 
		because the token at $a_0$ is moving along $P$ to $a_3$. Since the remaining $k-1$ tokens are fixed on $A\cap B$, $\PP$ is of Type 1.}	
	\label{figs:type1_case1}
\end{center}
\end{figure}

We now define certain paths of Type 2.
Let $e_1=a_1b_1$ and $e_2=a_2b_2$ be independent edges of 
$G$, and let $A,B\in F_k(G)$ be such that $A\setminus B=\{a_1,a_2\}$ and $B\setminus A=\{b_1,b_2\}$. A simple way to construct
an $A-B$ path $\RR$ of Type 2 (and length 2) is by moving the token at $a_1$ to $b_1$ along $e_1$, 
and then, by moving the token at $a_2$ to $b_2$ along $e_2$. We denote this construction by
\begin{displaymath}
	a_1\longrightarrow b_1;a_2\longrightarrow b_2.
\end{displaymath}

Then $\RR=A_0A_1A_2$, where $A_0=A$, $A_1=(A_0\setminus \{a_1\})\cup \{b_1\}$, 
$A_2=(A_1\setminus \{a_2\})\cup \{b_2\}=B$ (see Figure~\ref{figs:tentative-1}). We remark that $\RR$ can be seen as the
concatenation of two paths of Type 1, namely those corresponding to $a_1\longrightarrow b_1$ and $a_2\longrightarrow b_2$. 
As suggested above, we  use a semicolon `` ; '' to denote the concatenation of paths of Type 1. 

\begin{figure}[h!]
	\begin{center}	
	\includegraphics[width=0.62\textwidth]{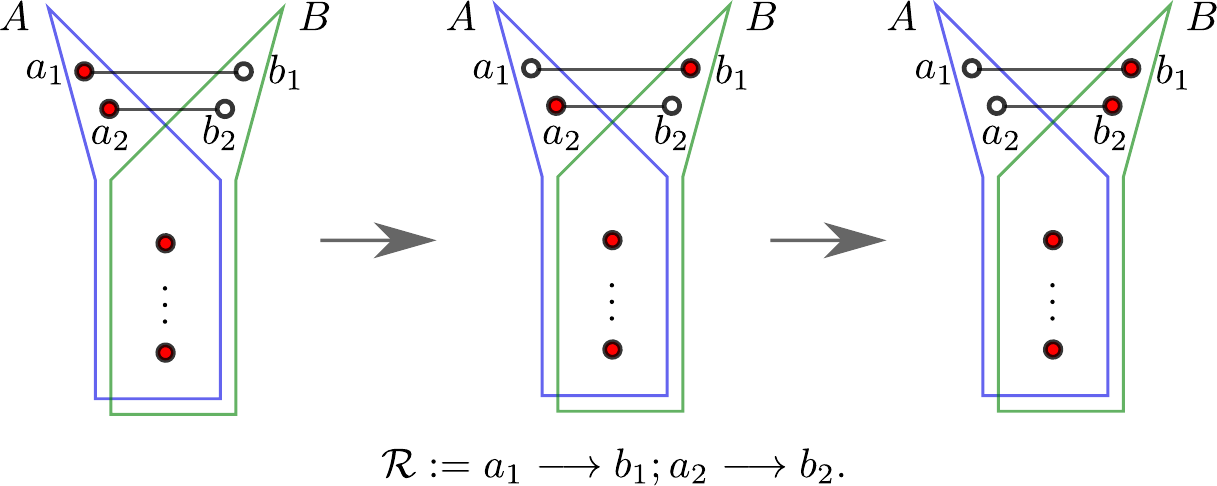}
	\caption{An $A-B$ path of Type 2.}
	\label{figs:tentative-1}
	\end{center}
\end{figure}

Now, suppose that $A$ and $B$ are adjacent vertices in $F_k(G)$ with 
$A\setminus B:=\{a\}$ and $B\setminus A:=\{b\}$. Then $ab$ is an edge of $G$.
Let $u$ and $v$ be adjacent vertices of $G$ such that $u\in A\cap B$ and $v\in V(G)\setminus(A\cup B)$. 
As we have seen above, a way to produce an $A-B$ path $\PP$ is simply by moving the token at $a$ to $b$ along the edge $ab$.
Now we use a simple trick, involving the edges $uv$ and $ab$, to produce a new 
$A-B$ path $\PP_{uv}$ of $F_k(G)$ that is internally disjoint from $\PP$. The path $\PP_{uv}$ is constructed as follows. 
First we move the token at $u$ to $v$ along $uv$, and then we move the token at $a$ to $b$ along $ab$, and finally 
we move back the token at $v$ to $u$ along $uv$. Clearly, each of these moves is admissible and they together define 
the required $\PP_{uv}$ path, which we denote by:   
\begin{displaymath}
	u\longrightarrow v; a\longrightarrow b; v\longrightarrow u.
\end{displaymath}

We say that the vertex $v$ is playing the role of a \emph{distractor}, which allow us to  produce  a new path 
$P_{uv}$ from $\PP$ and $uv$. See Figure~\ref{figs:distractor}.

We now generalize the above construction. Suppose that $\PP$ is an $A-B$ path of $F_k(G)$ and that $uv$ is an edge 
of $G$ with $u\in A\cap B$ and $v\in V(G)\setminus(A\cup B)$. 
If $u\in I$ and $v\notin I$ for any internal vertex $I$ of $\PP$, then we can get a new $A-B$ path
$P_{uv}$ from $\PP$ and $uv$ as follows. First move the token at $u$ to $v$ along $uv$. Then, keeping the 
token at $v$ fixed, move the tokens from the vertices in $A\setminus B$ to the vertices in $B\setminus A$ 
according to $\PP$, and finally move back the token at the distractor $v$ to the initial vertex $u$. Note that 
at the end we have produced an $A-B$ path $\PP_{uv}$ with the following property: for each inner vertex $J$ of 
$\PP_{uv}$, we have that $v\in J$ and $u\notin J$. This implies that if $u'v'$ is an edge of $G\setminus \{uv\}$ 
satisfying the same properties as $uv$ with respect to $\PP$, then the corresponding path $\PP_{u'v'}$ is an $A-B$ 
path internally disjoint from both $\PP$ and $\PP_{uv}$.
The paths produced in this way play an important role in the proof of Theorem~\ref{thm:main}.

\begin{figure}[h]
	\begin{center}
	\includegraphics[width=0.8\textwidth]{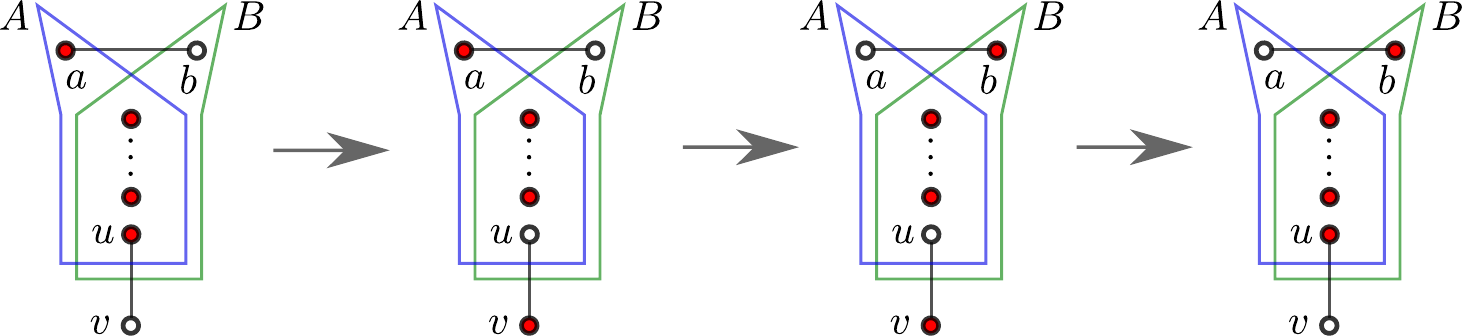}
	\caption{An $A-B$ path $\PP_{uv}$ with distractor $v$.}
	\label{figs:distractor}
\end{center}
\end{figure}


\subsection{Some basic facts}\label{sec:basicfacts}

In this section we prove auxiliary results that are used in the proof of Theorem~\ref{thm:main}.

\begin{prop}
	\label{prop:restriction}
	Let $H$ be a connected graph. Then $H$ is $t$-connected if and only if $H$ has 
	$t$ pairwise internally disjoint $a-b$ paths, for any two vertices $a$ and $b$ of $H$ such that $d_H(a,b)=2$.
\end{prop} 
\begin{proof} The forward implication follows directly from Menger's Theorem. Conversely, let $U$ be a vertex 
	cut of $H$ of minimum order. 	Let $H_1$ and $H_2$ be two distinct components  of $H-U$, and let $u\in U$. 
	Since $U$ is a minimum cut, then $u$ has at least a neighbour $v_i$ in $H_i$, for $i=1,2$. Then $d_H(v_1,v_2)=2$. 
	By hypothesis, $H$ has $t$ pairwise internally disjoint $v_1-v_2$ paths. Since each of these 
	$t$ paths intersects $U$, then we have that $|U|\geq t$, as required. 
\end{proof}


\begin{prop} \label{prop:cases}
	Let $X$ and $Y$ be vertices of $F_k(G)$ with $d_{F_k(G)}(X,Y)=2$. Then the following hold: 
	\begin{enumerate}[1{)}]
		\item $|X\cap Y|=k-1$ or $|X\cap Y|=k-2$,
		
		\item If $|X\cap Y|=k-2$, then $G$ has two independent edges $x_1y_1$ and $x_2y_2$ such that  
		$X\setminus Y=\{x_1,x_2\}$ and $Y\setminus X=\{y_1,y_2\}$. 
		
		\item If $|X\cap Y|=k-1$, then $G$ has two vertices $x$ and $y$ at distance two in $G$ such that 
		$X\setminus Y=\{x\}$ and $Y\setminus X=\{y\}$. 
	\end{enumerate}	
\end{prop}

\begin{proof} \-\ 
	
	\begin{enumerate}[1{)}]
		\item 
		This is equivalent to showing that $|X\triangle Y|\in \{2,4\}$. Since $X$ and $Y$ are distinct 
		$k$-sets of $V(G)$, $|X\triangle Y|$ must be an even positive integer.
		If $|X\triangle Y|\geq 6$, then we need to carry at least 3 tokens from the vertices in $X\setminus Y$ to the 
		vertices in $Y\setminus X$, and so $d_{F_k(G)}(X,Y)\geq 3$. Hence $|X\triangle Y|\in \{2,4\}$, as required. 
		See Figure~\ref{figs:obs_distancetwo}.
		
		\item 
		Note that $|X\setminus Y|=|Y\setminus X|=2$ in this case. Since $d_{F_k(G)}(X,Y)=2$, there is a way to 
		carry the two tokens at the vertices of $X\setminus Y$ to the vertices of $Y\setminus X$ with exactly two admissible token moves.
		These two token moves corresponds to two independent edges joining  vertices of $X\setminus Y$ with the vertices of $Y\setminus X$.
		See Figure~\ref{figs:obs_distancetwo} ($i$). 
		
		\item 
		In this case $X\setminus Y$  and $Y\setminus X$ each
		consists of exactly one vertex of $G$; say $x$ and $y$, respectively. Since $d_{F_k(G)}(X,Y)=2$, then $x$ cannot be adjacent to $y$ in $G$. 
		On the other hand, $d_{F_k(G)}(X,Y)=2$ implies the existence of an $X-Y$ path $\PP$ produced by exactly 2 admissible token moves.
		Now note that $\PP$ necessarily involves two admissible token moves $x\longrightarrow v$ and $u\longrightarrow y$.
		There are two possibilities either  $x\longrightarrow v$ is  applied before $u\longrightarrow y$
		or $u\longrightarrow y$   is applied
		before $x\longrightarrow v$. 
		Since $\PP$ is produced by exactly 2 admissible token moves, we have that $u=v\in N_G(x)\cap N_G(y)$, and $xvy$ is a 
		path of length two in $G$, as required.
		The two possibilities are depicted in ($ii$) and ($iii$) of Figure~\ref{figs:obs_distancetwo}.
	\end{enumerate}
\end{proof}

\begin{figure}[h]
	\begin{center}
	\includegraphics[width=0.8\textwidth]{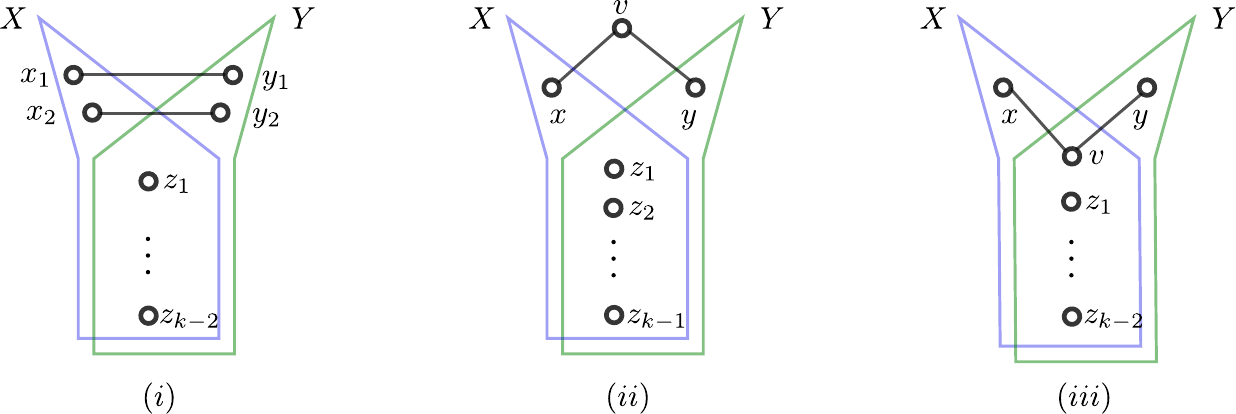}
	\caption{$X$ and $Y$ are vertices of $F_k(G)$ at distance $2$. ($i$) $X\triangle Y=\{x_1,y_1,x_2,y_2\}$
		and $x_1y_1, x_2y_2$ are independent edges of $G$. In ($ii$) and ($iii$) $X\triangle Y=\{x,y\}$ and $xvy$ is a shortest $x-y$ path in $G$. 
		The difference between the last two cases is that in ($ii$) $v\in V(G)\setminus (X\cup Y)$ and  in ($iii$) $v\in X\cap Y$. 
	}
	\label{figs:obs_distancetwo}
\end{center}
\end{figure}

Let $X$ be a vertex of $F_k(G)$. From the definition of $F_k(G)$ it is not hard to see that the \emph{complementary map} 
$\psi(X):=V(G)\setminus X$ defines an isomorphism between $F_k(G)$ and $F_{n-k}(G)$. The next proposition follows from 
the definition of $\psi$. 

\begin{prop}\label{prop:only(ii)}
	Let $\psi: F_k(G) \rightarrow F_{n-k}(G)$ be the complementary isomorphism, and let $X,Y,x,y$ and $v$ be as 
	in the proof of Proposition \ref{prop:cases} (3). Then exactly one of $v\notin X\cup Y $ or $v\notin \psi(X)\cup \psi(Y)$ holds. 
\end{prop}
\begin{proof}
	From  Proposition \ref{prop:cases} (3) we know that $\{x\}=X\setminus Y$ and  $\{y\}=Y\setminus X$. Since $P=xvy$ 
	is a path of length 2, then we have that $v\notin \{x,y\}$.
	These imply that exactly one of $v\in X\cap Y$ or $v\in V(G)\setminus (X\cup Y)$ holds. Since $v\in X\cap Y$ is equivalent 
	to $v\notin \psi(X)\cup \psi(Y)$, and 
	$v\in V(G)\setminus (X\cup Y)$ is equivalent to $v\notin X\cup Y$, we are done.   
\end{proof}


\section{Proof of Theorem \ref{thm:main}}\label{sec:proof}

Throughout this section, $T$ is a tree of order $n\geq 2$, and $k\in \{1,2,\ldots , n-1\}$. 
It is sufficient to show that 
\begin{displaymath}
	\kappa(F_k(T)) \ge \delta(F_k(T)).
\end{displaymath}
From the definition of $F_1(G)$ it is straightforward to see that $G$ and $F_1(G)$ are isomorphic.
In this case Theorem~\ref{thm:main} holds. We  assume that $n\geq 4$ and  $k\in \{2,\ldots ,n-2 \}$. 
By Proposition \ref{prop:restriction}, it suffices to prove the following.

\begin{lemma}\label{lem:horse}
	Let $X,Y\in V(F_k(T))$ with $d_{F_k(T)}(X,Y)=2$. Then 
	$F_k(T)$ has at least $\delta(F_k(T))$  pairwise internally disjoint $X-Y$ paths. 
\end{lemma}
\begin{proof}
	For brevity of notation, let $XY:=X\cap Y$, $\overline{XY}:=V(T)\setminus(X\cup Y),$ and $\delta:=\delta(F_k(T))$. 
	We remark that here $XY$ and $\overline{XY}$ are subsets of $V(T)$, but not edges of $F_k(T)$ or $F_{n-k}(T)$. 
	
	Informally, the general strategy to show Lemma~\ref{lem:horse} is as follows.  
	\begin{itemize} 
		\item \textsc{Step 1}. First, we construct a certain number $m$ of pairwise internally disjoint $X-Y$ paths in $F_k(T)$.
		\item \textsc{Step 2}. If $\delta > m$, we construct the $\delta-m$ missing $X-Y$ paths. 
	\end{itemize}
	
	The hypothesis $d(X,Y)=2$ and Proposition~\ref{prop:cases} (1) imply that $|XY|=k-1$ or $|XY|=k-2$. 
	We analyze these cases separately.

	\subsection{\textsc{Case 1:} $|XY|=k-1$} 
	
	From Proposition~\ref{prop:cases} (3) we know that there exist $x,y,v\in V(T)$ such that
	$\{x\}=X\setminus Y, \{y\}=Y\setminus X, v\notin \{x,y\}$, and $P=xvy$ is a shortest $x-y$ path of $T$. 
	In view of Proposition~\ref{prop:only(ii)}, we can assume without any loss of generality that $v\notin X\cup Y$. 
	Indeed, if  $v\in X\cup Y$ then by Proposition~\ref{prop:only(ii)}
	$v\notin \psi(X)\cup \psi(Y)$. Since $F_k(T)$ and $F_{n-k}(T)$ are isomorphic under $\psi(U)=V(T)\setminus U$,
	then we can work with $\psi(X)$ and $\psi(Y)$ in $F_{n-k}(T)$ instead of $X$ and $Y$ in $F_{k}(T)$. 
	We assume that $X$ and $Y$ are as in Figure~\ref{figs:obs_distancetwo}~($ii$). 
	Let $\overline{XY}':=\overline{XY}\setminus \{v\}$ and let
	
	\begin{align*}
		\overline{XY}(x)&:=\{w\in \overline{XY}':\text{$w$ is adjacent to $x$}\}=\{w_x^1,\ldots, w_x^a\}, \\
		\overline{XY}(y)&:=\{w\in \overline{XY}':\text{$w$ is adjacent to $y$}\}=\{w_y^1,\ldots,w_y^d\}, \\
		XY(x)&:=\{z\in XY:\text{$z$ is adjacent to $x$}\}=\{z_x^1,\ldots,z_x^c\},  \\
		XY(y)&:=\{z\in XY:\text{$z$ is adjacent to $y$}\}=\{z_y^1,\ldots,z_y^b\},
	\end{align*}
	where $a:=|\overline{XY}(x)|$, $b:=|XY(y)|$, $c:=|XY(x)|$, and $d:=|\overline{XY}(y)|$. See Figure~\ref{fig:case1}. 
	
	\begin{figure}[h]
		\begin{center}
		\includegraphics[width=0.4\textwidth]{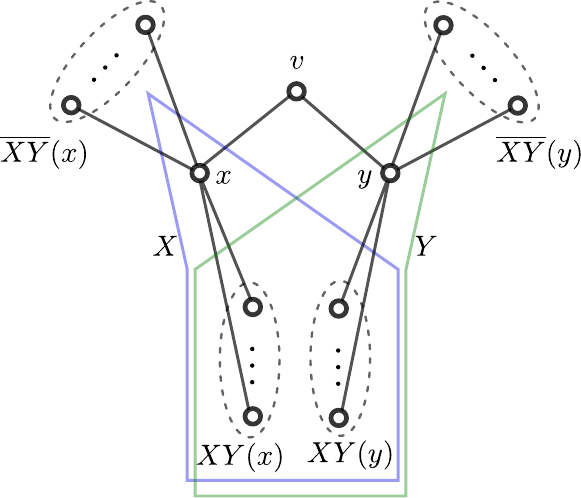}
		\caption{The neighbors of $x$ and $y$ in \textsc{Case 1}.}
		\label{fig:case1}
		\end{center} 
	\end{figure}

	Let us define 
	\begin{displaymath}
		E_{XY,\overline{XY}}:=\{zw\in E(T) : z\in XY \text{ and }w\in \overline{XY}\},  \text{ and } \eta:=|E_{XY,\overline{XY}}|.
	\end{displaymath}
	Since $T$ is a tree, then $\overline{XY}(x), \overline{XY}(y), XY(x),$ and $XY(y)$ are pairwise disjoint. Then, 
	in $F_k(T)$, $deg(X)=a+b+\eta+1$ and $deg(Y)=c+d+\eta+1.$ Without loss of generality we may assume that $deg(X)\leq deg(Y)$.
	Hence, $a+b\leq c+d$. 
	
	Let $m_x:=\min\{a,c\}$, $m_y:=\min\{b,d\}$, and $m:=m_x+m_y+\eta+1.$
	
	
	\subsubsection{\textsc{Step 1} of \textsc{Case 1}}

	We produce the required $m$ $X-Y$ paths by means of four types of constructions. 
	
	\begin{enumerate}
		\item Using the vertex $v$: 
		\begin{displaymath}
			\PP_0:=x\longrightarrow v\longrightarrow y.
		\end{displaymath}
		Let $\mathbb{T}_1:=\{\PP_0\}$. Note that if $A_0$ is the (unique) inner vertex of $\PP_0$, then
		
		\begin{property}
			\item[\enspace(C1)]  $A_0\cap XY=XY$ and $A_0\cap \overline{XY}'=\emptyset$.
		\end{property}
		
		\item Using the edges of $E_{XY,\overline{XY}}$. For each $z_iw_j\in E_{XY,\overline{XY}}$, 
		let $\PP_{i,j}$ be the $X-Y$ path defined as follows: 
		\begin{displaymath}
		\PP_{i,j}:=\begin{cases}
			z_i\rightarrow w_j; x\rightarrow v\rightarrow y; w_j\rightarrow z_i & \text{if $w_j\neq v$;} \\
			z_i\rightarrow v\rightarrow y; x\rightarrow v\rightarrow z_i & \text{if $w_j=v$.}
		\end{cases}
		\end{displaymath}
		Let $\mathbb{T}_2:=\{\PP_{i,j}:z_iw_j\in E_{XY,\overline{XY}}\}$. 
		Note that if $A_{i,j}$ is an inner vertex of $\PP_{i,j}$, then  
		
\begin{property}
	\item[\enspace(C2)] $A_{i,j}\cap XY=XY\setminus \{z_i\}$. 
\end{property}
		Moreover, depending on whether $w_j\neq v$ or $w_j=v$, then $A_{i,j}$ also satisfies the following:
		
		\begin{property}
			\item[\enspace(C2.1)] If $w_j\neq v$, then  $A_{i,j}\cap \overline{XY}'=\{w_j\}$. 
			\item[\enspace(C2.2)] If $w_j=v$, then $A_{i,j}\cap \overline{XY}'=\emptyset$. 
		\end{property}

		We recall that if $r=0$, then $[r]=\emptyset$. 
		
		\item Using the vertices $w_x^i\in \overline{XY}(x)$ and $z_x^i\in XY(x)$. For each $i \in [m_x]$, we define the path $\PP_i$ as follows: 
		\begin{displaymath}
			\PP_i:=x\rightarrow w_x^i; z_x^i\rightarrow x \rightarrow v \rightarrow y;w_x^i\rightarrow x\rightarrow z_x^i. 
		\end{displaymath}
		Let $\mathbb{T}_3:=\{\PP_i:i\in [m_x]\}$. Again, note that if $A_i$ is an inner vertex of $\PP_i$, then

		\begin{property}
			\item[\enspace(C3)] Either $A_i\cap XY=XY$ or $A_i\cap XY=XY\setminus \{z_{x}^i\}$, and either $A_i\cap \overline{XY}'=\emptyset$
			or $A_i\cap \overline{XY}'=\{w_{x}^i\}$, and at least one of the following holds: $A_i\cap XY=XY\setminus \{z_x^i\}$ 
			or $A_i\cap \overline{XY}'=\{w_{x}^i\}$. 
		\end{property}
		
		\item Using the vertices $w_y^j\in \overline{XY}(y)$ and $z_y^j\in XY(y)$. For each $j \in [m_y]$, we define the path $\QQ_j$ as follows:
		\begin{displaymath}
			\QQ_j:= z_y^j\rightarrow y \rightarrow w_y^j; x\rightarrow v\rightarrow y\rightarrow z_y^j;w_y^j\rightarrow y. 
		\end{displaymath}
		Let $\mathbb{T}_4:=\{\QQ_j:j\in [m_y]\}$. Again, note that if $A_j$ is an inner vertex of $\QQ_j$, then
		
		\begin{property}
			\item[\enspace(C4)] Either $A_j\cap XY=XY$ or $A_j\cap XY=XY\setminus \{z_{y}^j\}$, and either $A_j\cap \overline{XY}'=\emptyset$
			or $A_j\cap \overline{XY}'=\{w_{y}^j\}$, and at least one of the following holds: $A_j\cap XY=XY\setminus \{z_y^j\}$
			or $A_j\cap \overline{XY}'=\{w_{y}^j\}$. 
		\end{property}
	\end{enumerate}
	
	Let us define  $\mathbb{T}:=\mathbb{T}_1\cup \mathbb{T}_2\cup \mathbb{T}_3\cup \mathbb{T}_4.$ 
	Since $|\mathbb{T}_1|=1, |\mathbb{T}_2|=\eta, |\mathbb{T}_3|=m_x, |\mathbb{T}_4|=m_y$, and 
	$m=1+\eta+m_x+m_y$, then in order to finish the \textsc{Step~1} of \textsc{Case~1}, it is enough to 
	show that the paths in $\mathbb{T}$ are pairwise internally disjoint. 
	\vskip 0.2cm
	\begin{claim}\label{cl:first_constructions}
		The $X-Y$ paths in $\mathbb{T}$ are pairwise internally disjoint. 
	\end{claim}
\begin{prf}{\ref{cl:first_constructions}}First we show separately that the paths in $\mathbb{T}_{\ell}$ 
	are pairwise internally disjoint for $\ell\in \{2,3,4\}$. 
	
	Suppose that $\boldsymbol{\ell=2}$, and let $\PP_{i,j}$ and $\PP_{s,t}$ be distinct paths in $\mathbb{T}_2$. 
	Let $A_{i,j}$ and $A_{s,t}$ be inner vertices of $\PP_{i,j}$ and $\PP_{s,t}$, respectively. 
	Since $(i,j)\neq (s,t)$, then $z_i\neq z_s$ or $w_j\neq w_t$. 
	
	If $z_i\neq z_s$, then from (C2) we know that $A_{i,j}\cap XY=XY\setminus\{z_i\}$ and $A_{s,t}\cap XY=XY\setminus\{z_s\}$. Hence 
	$z_i\in A_{s,t}\setminus A_{i,j}$, which implies that $A_{i,j}\neq A_{s,t}$. 
	
	Now suppose that $w_j\neq w_t$. First suppose that $v\notin \{w_j, w_t\}$. By (C2.1) we have $A_{i,j}\cap \overline{XY}'=\{w_j\}$, 
	and similarly, $A_{s,t}\cap \overline{XY}'=\{w_t\}$. Then, $A_{i,j}\cap \overline{XY}'\neq A_{s,t}\cap \overline{XY}'$ 
	and so $A_{i,j}\neq A_{s,t}$. Then	we may assume that $v\in \{w_j, w_t\}$. Without loss of generality suppose that 
	$w_j=v$. We know by (C2.2) that	$A_{i,j}\cap \overline{XY}'=\emptyset$, and by (C2.1) that $A_{s,t}\cap \overline{XY}'=\{w_t\}$, 
	these two facts imply that  $A_{i,j}\neq A_{s,t}$. 
	
	Suppose that $\boldsymbol{\ell=3}$, and let $\PP_s$ and $\PP_t$  be distinct paths in $\mathbb{T}_3$. 
	For $r\in \{s,t\}$, let $A_r$ be an inner vertex of $\PP_r$. From the last assertion of (C3) we know that 
	$A_s\cap \overline{XY}'=\{w_x^s\}$ or $A_s\cap XY=XY\setminus \{z_x^s\}$. Suppose that $A_s\cap \overline{XY}'=\{w_{x}^s\}$. 
	Since (C3) implies that $A_t\cap \overline{XY}'=\emptyset$ or $A_t\cap \overline{XY}'=\{w_x^t\}$, then we have 
	$A_s\cap \overline{XY}'\neq A_t\cap \overline{XY}'$, and so $A_s\neq A_t$.  
	Now suppose that $A_s\cap XY=XY\setminus \{z_x^s\}$. Again, from (C3) we know that $A_t\cap XY=XY$ or 
	$A_t\cap XY=XY\setminus \{z_x^t\}$. Since $z_{x}^s\neq z_x^t$, then  $A_s\cap XY\neq A_t\cap XY$, and so $A_s\neq A_t$.

	Suppose that $\boldsymbol{\ell=4}$. This case can be handled in a totally analogous manner as previous case.	
	
	Let $A_0, A_{i,j}, A_s$, and $A_t$ be inner vertices of $\PP_0\in \mathbb{T}_1$, $\PP_{i,j}\in \mathbb{T}_2$, 
	$\PP_s\in \mathbb{T}_3$, and $\QQ_t\in \mathbb{T}_4$, respectively. It remains to show that $\PP_0,\PP_{i,j},\PP_s$, 
	and $\QQ_t$ are pairwise internally disjoint. We analyze separately each pair.
	
	\begin{mylist}
		\item[{${\{A_0, A_{i,j}\}}$:}] Here we have $A_0\cap XY=XY$, while
		$A_{i,j}\cap XY=XY\setminus \{z_i\}$, and so $A_0\neq A_{i,j}$. 
		
		\item[{${\{A_0, A_s\}}$:}] By (C1) we know that $A_0\cap XY=XY$ and that $A_0\cap \overline{XY}'=\emptyset$. 
		Similarly, by the last assertion of (C3), we know that either 
		$A_s\cap XY= XY\setminus\{z_x^s\}$ or $A_s\cap \overline{XY}'=\{w_x^s\}$, then we have $A_0\neq A_s$. 
		
		\item[{${\{A_0, A_t\}}$:}] As in previous case, the last assertion of (C4) implies that either 
		$A_t\cap XY= XY\setminus \{z_y^t\}$ or $A_t\cap \overline{XY}'= \{w_y^t\}$. Then, since   
		$A_0\cap XY=XY$ and $A_0\cap \overline{XY}'=\emptyset$, we have $A_0\neq A_t$.
		
		\item[{${\{A_{i,j}, A_s\}}$:}] First suppose that $w_j=v$. Then
		$z_i\neq z_x^s$, as otherwise the vertex set $\{x,z_i,v\}$ forms a cycle, contradicting
		that $T$ is a tree. Since $A_{i,j}\cap XY=XY\setminus \{z_i\}$, and either $A_s\cap XY=XY$ or $A_s\cap XY=XY\setminus \{z_x^s\}$,
		then $A_{i,j}\cap XY\neq A_s\cap XY$, as required.
		
		Suppose now that $w_j\neq v$. By (C3) we know that $A_s\cap XY=XY$ or $A_s\cap XY=XY\setminus \{z_x^s\}$. 
		If $A_s\cap XY=XY$, then $A_{i,j}\cap XY=XY\setminus \{z_i\}$ implies that  $A_s\neq A_{i,j}$. 
		Thus we may assume that $A_s\cap XY=XY\setminus \{z_{x}^s\}$. 
		If  $z_i\neq z_x^s$, then $XY\setminus \{z_{x}^s\}=A_s\cap XY\neq A_{i,j}\cap XY=XY\setminus \{z_i\}$, as desired.  
		Then we can assume that $z_x^s=z_i$. This implies that $w_x^s\neq w_j$, as otherwise $\{z_i,x,w_j\}$ forms a cycle. 
		By (C2.1) we know that $A_{i,j}\cap \overline{XY}'=\{w_j\}$, and by (C3) we have that either $A_s\cap \overline{XY}'=\emptyset$ or 
		$A_s\cap \overline{XY}'=\{w_x^s\}$. Since $w_x^s\neq w_j$, then $A_{i,j}\cap \overline{XY}'\neq A_s\cap \overline{XY}'$, as required.  
		
		\item[{${\{A_{i,j}, A_t\}}$:}] Again, this case can be handled in a totally analogous manner as previous case.	
		
		\item[{${\{A_s, A_t\}}$:}]  Since $XY(x), XY(y), \overline{XY}(x)$, and $\overline{XY}(y)$ are pairwise disjoint, then 
		$z_x^s\neq z_y^t$ and $w_x^s\neq w_y^t$. From these inequalities and (C3)-(C4) we have that either $A_s\cap XY\neq A_t\cap XY$
		or $A_s\cap \overline{XY}'\neq A_t\cap \overline{XY}'$, and so $A_s \neq A_t$. 	
	\end{mylist}
	This completes the proof of Claim \ref{cl:first_constructions}. \end{prf}

	
	\subsubsection{\textsc{Step 2} of \textsc{Case 1}} 
	
	We start by showing that $\delta - m \leq 2$. 
	
	\vskip 0.2cm
	
	\begin{claim}\label{claim:min-deg}
		Let $\delta, m, m_x, m_y,$ and $\eta$ be as above. Then,	
		\begin{displaymath}
		\delta \leq \begin{cases}
			m_x+m_y+\eta +1=m & \text{if $a\leq c$ and $b\leq d$, or $a>c$,}\\
			m_x+m_y+\eta +2=m+1 & \text{if $b=d+1$,} \\
			m_x+m_y+\eta+3=m+2 & \text{if $b\geq d+2$.}
		\end{cases}
		\end{displaymath}
	\end{claim}
\begin{prf}{\ref{claim:min-deg}}
First we note that  if $a\leq c$ and $b\leq d$, then
	\begin{displaymath}
		\delta \leq deg(X)=a+b+\eta +1=m_x+m_y+\eta+1=m,
	\end{displaymath}
	as claimed. 
	
	Suppose that $a>c$. Since $a+b\leq c+d$, then $b<d$. 
	Let $U:=\overline{XY}\cup \{x,y\}$. Since $T[U]$ is a forest, then it contains at least a vertex $u\in U\setminus \{v\}$
	such that $deg_{T[U]}(u)\leq 1$.  Note that $u\notin \{x,y\}$, because $deg_{T[U]}(x)=a+1\geq 2$ and 
	$deg_{T[U]}(y)=d+1\geq 2$. Let $X':=(X\setminus\{x\})\cup \{u\}$, so 
	\begin{displaymath}
		\delta\leq deg(X')\leq b+c+\eta+deg_{T[U]}(u)\leq m_x+m_y+\eta+1=m,
	\end{displaymath}
	as claimed. 
	
	Suppose that $b=d+1$. Since $a+b\leq c+d$, then $a<c$. In this case we have that 
	\begin{displaymath}
		\delta\leq deg(X)=a+b+\eta+1=a+(d+1)+\eta+1=m_x+m_y+\eta+2=m+1.
	\end{displaymath}
	
	Finally, suppose that $b\geq d+2$. Since $a+b\leq c+d$, then $c\geq a+2$. Let $U:=X\cup Y$. Since $T[U]$ is
	a forest,  then it contains at least a vertex $u\in U$ such that $deg_{T[U]}(u)\leq 1$. 
	Note that $u\notin\{x,y\}$, because $deg_{T[U]}(x)\geq c\geq 2$ and $deg_{T[U]}(y)\geq b \geq 2$.  
	Let $X'=(X\setminus\{u\})\cup \{y\}$, then 
	\begin{displaymath}
		\delta\leq deg(X')\leq (a+1)+(d+1)+\eta+deg_{T[U]}(u)\leq m_x+m_y+\eta+3=m+2.
	\end{displaymath}
	This completes the proof of Claim~\ref{claim:min-deg}. 
	\end{prf}
	
	Claim~\ref{claim:min-deg} shows that almost all $X-Y$ paths claimed by Lemma~\ref{lem:horse} are provided by $\mathbb{T}$, when $|XY|=k-1$.
	We finish the proof of \textsc{Case 1} with the construction of the remaining $\delta-m$ $X-Y$ paths. 
	\vskip 0.2cm
	\begin{claim}\label{claim:end-case1}
		If  $|XY|=k-1$, then $F_k(T)$ has at least $\delta$  $X-Y$ pairwise internally disjoint paths.
	\end{claim}
\begin{prf}{\ref{claim:end-case1}}
We have already constructed $m$ $X-Y$ pairwise internally disjoint paths, namely the elements of
	$\mathbb{T}$. Then, it remains to show the existence of $\delta-m$ additional $X-Y$ paths with similar properties.  
	Since if $\delta \leq m$ then there is nothing to prove,
	we assume that $\delta > m$. From this and Claim~\ref{claim:min-deg} it follows that $b\geq d+1$. Moreover, since $a+b\leq c+d$, then $c\geq a+1$.
	Hence, $a=\min\{a,c\}$ and $d=\min\{b,d\}$.
	
	Suppose first that $b=d+1$. By Claim~\ref{claim:min-deg} we have that $\delta \leq m+1$. Thus, it is enough to construct a new $X-Y$ 
	path internally disjoint to each path in $\mathbb{T}$. Since $b=d+1>d=\min\{b,d\}$ and $c\geq a+1>a=\min\{a,c\}$, then the vertices 
	$z_y^{b}$ and $z_x^{c}$ were not used in the construction of the paths of $\mathbb{T}_3\cup \mathbb{T}_4$. We construct the required 
	path $\PP$ as follows: 
	\begin{displaymath}
		\PP:=z_y^{b}\longrightarrow y; x\longrightarrow v; z_x^{c}\longrightarrow x; y\longrightarrow z_y^{b}; v\longrightarrow y; x\longrightarrow z_x^{c}.
	\end{displaymath}
	Let $A$ be an inner vertex of $\PP$. From the definition of $\PP$ it follows that
	
\begin{property}
	\item[\enspace(C5)] Either $A\cap XY=XY\setminus \{z_y^b\}$ or $A\cap XY=XY\setminus \{z_x^c\}$ or  $A\cap XY=XY\setminus \{z_x^b, z_y^c\}$, and that $A\cap \overline{XY}'=\emptyset$.
\end{property}

	Now we show that $\PP$ is internally disjoint to any path in $\mathbb{T}$. Let $A_0, A_{i,j}, A_s$, and $A_t$ be inner vertices of 
	$\PP_0\in \mathbb{T}_1$, $\PP_{i,j}\in \mathbb{T}_2$, $\PP_s\in \mathbb{T}_3$, and $\QQ_t\in \mathbb{T}_4$, respectively.		
	
	We analyze these cases separately. 		 
	\begin{mylist}
		\item[{${\{A_0, A\}}$:}]  
		By (C1) and (C5) we know that $A_0\cap XY=XY$ and $A\cap XY\neq XY$, respectively, and so $A\neq A_0$.  
		\item[{${\{A_{i,j}, A\}}$:}] 
		If $w_j\neq v$, then $A_{i,j}\cap \overline{XY}'=\{w_j\}$, and then $A\cap \overline{XY}'\neq A_{i,j}\cap \overline{XY}'$, 
		which implies that $A\neq A_{i,j}$. \\
		Now suppose that $w_j=v$. Then $z_i\notin \{z_y^{b}, z_x^{c}\}$, as otherwise $T$ has a cycle. 
		Then, by (C2) and (C5) we have that $A_{i,j}\cap XY\neq A\cap XY$,
		and so $A\neq A_{i,j}$. 
		
		\item[{${\{A_s, A\}}$:}] 
		Note that $z_x^s\neq z_x^{c}$, because $s\leq a<c$. Similarly, $z_x^s \neq z_y^{b}$, because 
		$XY(x)\cap XY(y)=\emptyset$. Then, (C3) and (C5) implies that $A_s\cap XY\neq A\cap XY$, and so $A\neq A_{s}$. 
		
		\item[{${\{A_t, A\}}$:}] 
		We proceed as in previous case. Since $t\leq d<b$, then $z_y^t\neq z_y^b$, and
		$z_y^t\neq z_x^{c}$ because $XY(x)\cap XY(y)=\emptyset$. 
		Then, (C4) and (C5) implies that $A_t\cap XY\neq A\cap XY$, and so $A\neq A_{t}$.   
	\end{mylist}

	Finally, suppose that $b\geq d+2$. By Claim~\ref{claim:min-deg} we have that $\delta \leq m+2$. Thus, it is 
	enough to construct two $X-Y$ 	paths, say  $\PP$ and $\PP'$, such that $\{\PP,\PP'\}\cup \mathbb{T}$ is a set 
	of pairwise internally disjoint paths. 
	
	Since $b\geq d+2$ and $a+b\leq c+d$, then $c\geq a+2$. Now we use $z_y^{b},z_y^{b-1},z_x^{c}$, and $z_x^{c-1}$ to construct
	$\PP$ and $\PP'$ as follows.
	\begin{displaymath}
		\PP:=z_y^{b}\rightarrow y; x\rightarrow v; z_x^{c}\rightarrow x; y\rightarrow z_y^{b}; v\rightarrow y; x\rightarrow z_x^{c}, \text{ and }
	\end{displaymath}
	\begin{displaymath}
		\PP':=z_y^{b-1}\rightarrow y; x\rightarrow v; z_x^{c-1}\rightarrow x; y\rightarrow z_y^{b-1}; v\rightarrow y;x\rightarrow z_x^{c-1}.
	\end{displaymath}
	
	Note that a similar argument to the one used above (for the case $b=d+1$) can be applied to show that  
	$\PP$ and $\PP'$ are internally disjoint of  
	each path in $\mathbb{T}$. Hence all that remains to be checked is that $\PP$ and $\PP'$ are  internally disjoint. 
	
	Let $A$ and $A'$ be inner vertices of $\PP$ and $\PP'$, respectively. From the definition of $\PP$ 
	(respectively, $\PP'$) we know that either 	$A\cap XY=XY\setminus\{z_y^{b}\}$, $A\cap XY=XY\setminus\{z_x^{c}\}$, 
	or $A\cap XY=XY\setminus \{z_y^{b},z_x^{c}\}$ (respectively, $A'\cap XY=XY\setminus \{z_y^{b-1}\}$, 
	$A'\cap XY=XY\setminus\{z_x^{c-1}\}$, or $A'\cap XY=XY\setminus\{z_y^{b-1},z_x^{c-1}\}$). Since 
	$\{z_y^{b},z_x^{c}\}\cap \{z_y^{b-1},z_x^{c-1}\}=\emptyset$, then in all the arising cases, we always have $A\neq A'$, as required. 
	This completes the proof of Claim~\ref{claim:end-case1}, and hence the proof of \textsc{Case 1}. \end{prf}


	\subsection{\textsc{Case 2:} $|XY|=k-2$} 
	
	From Proposition~\ref{prop:cases} (2) we know that $T$ has two independent edges 
	$x_1y_1$ and $x_2y_2$ such that $X\setminus Y=\{x_1,x_2\}$ and $Y\setminus X=\{y_1,y_2\}$. 
	Then, we can assume that $X$ and $Y$ are as in Figure~\ref{figs:obs_distancetwo}~($i$).  
	Similarly as in \textsc{Case 1}, for $i\in \{1,2\}$, let us define
	
	\begin{align*}
		\overline{XY}(x_i)&:=\{w\in \overline{XY}:w\text{ is adjacent to } x_i\}=\{w_{x_i}^1,\ldots, w_{x_i}^{a_i}\}, \\
		\overline{XY}(y_i)&:=\{w\in \overline{XY}:w\text{ is adjacent to } y_i\}=\{w_{y_i}^1,\ldots, w_{y_i}^{d_i}\}, \\ 
		XY(x_i)&:=\{z\in XY:z\text{ is adjacent to } x_i\}=\{z_{x_i}^1,\ldots, z_{x_i}^{c_i}\}, \\
		XY(y_i)&:=\{z\in XY:z\text{ is adjacent to } y_i\}=\{z_{y_i}^1,\ldots, z_{y_i}^{b_i}\}, 
	\end{align*}
	where $a_i:=|\overline{XY}(x_i)|$, $b_i:=|XY(y_i)|$, $c_i:=|XY(x_i)|$, and $d_i:=|\overline{XY}(y_i)|$.  
	
	The next observation follows easily from the involved definitions and the fact that $T$ is a tree. 
	
	\begin{obs}\label{obs:intersections} Let $i\in \{1,2\}$. Then $\overline{XY}(x_i)\cap \overline{XY}(y_i)=\emptyset$ 
		and  $XY(x_i)\cap XY(y_i)=\emptyset$, and  at most one of the following occurs: 
		$|\overline{XY}(x_1)\cap \overline{XY}(x_2)|=1$, $|\overline{XY}(y_1)\cap \overline{XY}(y_2)|=1$, 
		$|XY(x_1)\cap XY(x_2)|=1$, or $|XY(y_1)\cap XY(y_2)|=1$. 
	\end{obs}
	
	Let us define
	\begin{displaymath}
		E_{XY,\overline{XY}}:=\{z_iw_j\in E(G):z_i\in XY \text{ and } w_j\in \overline{XY}\}, \text{ and let }\eta:=|E_{XY,\overline{XY}}|.
	\end{displaymath}
	Then
\begin{displaymath}
	deg(X)=
	\begin{cases}
		a_1+a_2+b_1+b_2+\eta+2 & \text{ if $x_1y_2\notin E(T)$ and $x_2y_1\notin E(T)$,}\\
		a_1+a_2+b_1+b_2+\eta+3 & \text{ otherwise. } 
	\end{cases}
\end{displaymath}
	and, 
\begin{displaymath}
	deg(Y)=
	\begin{cases}
		c_1+c_2+d_1+d_2+\eta+2 & \text{ if $x_1y_2\notin E(T)$ and $x_2y_1\notin E(T)$, }\\
		c_1+c_2+d_1+d_2+\eta+3 & \text{ otherwise. } 
	\end{cases}
\end{displaymath}
	
	Note that the term ``+3" in $deg(X)$ and $deg(Y)$ means that $T$ has 3 edges with an end in $\{x_1,x_2\}$ 
	and the other end in $\{y_1,y_2\}$. 
	Then it is impossible to have $deg(X)=a_1+a_2+b_1+b_2+\eta+2$ 
	and $deg(Y)=c_1+c_2+d_1+d_2+\eta+3$ simultaneously. Similarly,  $deg(X)=a_1+a_2+b_1+b_2+\eta+3$ 
	and $deg(Y)=c_1+c_2+d_1+d_2+\eta+2$ cannot occur simultaneously. 
	
	Without loss of generality we assume that  $deg(X)\leq deg(Y)$.  This assumption together with the assertions 
	of the previous paragraph imply that $a_1+a_2+b_1+b_2\leq c_1+c_2+d_1+d_2$. For $i\in \{1,2\}$, let 
	$m_{x_i}:=\min\{a_i,c_i\}, m_{y_i}:=\min\{b_i,d_i\},$ and $m:=m_{x_1}+m_{x_2}+m_{y_1}+m_{y_2}+\eta+2.$
	
	
	\subsubsection{\textsc{Step 1} of \textsc{Case 2}}
	
	We proceed similarly as in \textsc{Case 1}. In particular, we often use slight adaptation of many arguments given in \textsc{Case 1}.
	We start by producing $m$ $X-Y$ paths by means of six types of constructions. 
	
	\begin{enumerate}
		\item Let us define $\PP_{x_1}$ and $\PP_{x_2}$ as follows:
		\begin{displaymath}
			\PP_{x_1}:=x_1\rightarrow y_1; x_2\rightarrow y_2
		\end{displaymath}
		\begin{displaymath}
			\PP_{x_2}:=x_2\rightarrow y_2; x_1\rightarrow y_1.
		\end{displaymath}
		Let $\ML_1:=\{\PP_{x_1},\PP_{x_2}\}$. Let $\PP\in \ML_1$, and let $A$ be an inner vertex of $\PP$. Then
		
		\begin{property}
			\item[\enspace(D1)] $A\cap XY=XY$ and $A\cap \overline{XY}=\emptyset$. 
		\end{property}
		
		\item For each edge $z_iw_j\in E_{XY,\overline{XY}}$, let
		\begin{displaymath}
			\PP_{i,j}:= z_i\rightarrow w_j; x_1\rightarrow y_1; x_2\rightarrow y_2; w_j\rightarrow z_i.
		\end{displaymath}
		Let $\ML_2:=\{\PP_{i,j}:z_iw_j\in E_{XY,\overline{XY}}\}.$ Let $\PP_{i,j}\in \ML_2$, and let $A_{i,j}$ be an inner 
		vertex of $\PP_{i,j}$. Then 
		\begin{property}
			\item[\enspace(D2)] $A_{i,j}\cap XY=XY\setminus \{z_i\}$ and $A_{i,j}\cap \overline{XY}=\{w_j\}$. 
		\end{property}
		
		\item For each $i\in [m_{x_1}]$, we define the path $\PP_i$ as follows: 
		\begin{displaymath}
			\PP_i:= x_1\longrightarrow w_{x_1}^i; z_{x_1}^i\longrightarrow x_1\longrightarrow y_1;  x_2\longrightarrow y_2;
		w_{x_1}^i\longrightarrow x_1\longrightarrow z_{x_1}^i. \end{displaymath}
		Let $\ML_3:=\{\PP_i:i\in[m_{x_1}]\}$. Let $\PP_i\in \ML_3$, and let $A_i$ be an inner vertex of $\PP_i$. Then 
		\begin{property}
			\item[\enspace(D3)] Either $A_i\cap XY=XY$ or $A_i\cap XY=XY\setminus\{z_{x_1}^i\}$, and either $A_i\cap \overline{XY}=\emptyset$ or $A_i\cap \overline{XY}=\{w_{x_1}^i\}$, 
			and at least one of the following holds: $A_i\cap \overline{XY}=\{w_{x_1}^i\}$ or $A_i\cap XY=XY\setminus\{z_{x_1}^i\}$. 
			
		\end{property}
		
		\item For each $j\in [m_{y_1}]$, we define the path $\QQ_j$ as follows: 
\begin{displaymath}
	\QQ_j:= z_{y_1}^j\longrightarrow y_1\longrightarrow w_{y_1}^j; x_2\longrightarrow y_2; 
	x_1\longrightarrow y_1\longrightarrow z_{y_1}^j; w_{y_1}^j\longrightarrow y_1. 
\end{displaymath}
		Let  $\ML_4:=\{\QQ_j:j\in[m_{y_1}] \}$. Let $\QQ_j\in \ML_4$, and let $A_j$ be an inner vertex of $\QQ_j$. Then
		\begin{property}
			\item[\enspace(D4)] Either $A_j\cap XY=XY\setminus \{z_{y_1}^j\}$ or $A_j\cap XY=XY$, and either $A_j\cap \overline{XY}=\emptyset$
			or $A_j\cap \overline{XY}=\{w_{y_1}^j\}$, and at least one of the following holds: $A_j\cap XY=XY\setminus \{z_{y_1}^j\}$
			or $A_j\cap \overline{XY}=\{w_{y_1}^j\}$.  
		\end{property}
		
		\item For each $i\in [m_{x_2}]$, we define $\PP^*_i$  as follows: 
		\begin{displaymath}
			\PP^*_i:=x_2\longrightarrow w_{x_2}^i; z_{x_2}^i\longrightarrow x_2\longrightarrow y_2; x_1\longrightarrow y_1;
		w_{x_2}^i\longrightarrow x_2\longrightarrow z_{x_2}^i. \end{displaymath}
		Let $\ML^*_3:=\{\PP^*_i: i\in[m_{x_2}]\}$. Let $\PP^*_i\in \ML_3^*$, and let $A^*_i$ be an inner vertex of $\PP^*_i$. 
		Then
		\begin{property}
			\item[\enspace(D3*)] Either $A^*_i\cap XY=XY$ or $A_i^*\cap XY=XY\setminus\{z_{x_2}^i\}$, and either 
			$A_i^*\cap \overline{XY}=\emptyset$ or $A_i^*\cap \overline{XY}=\{w_{x_2}^i\}$, and at least one of the following holds: 
			$A_i^*\cap \overline{XY}=\{w_{x_2}^i\}$ or $A_i^*\cap XY=XY\setminus\{z_{x_2}^i\}$. 
		\end{property}

		\item For each $j\in [m_{y_2}]$, we define $\QQ^*_j$ as follows:
		\begin{displaymath}
			\QQ^*_j:= z_{y_2}^j\longrightarrow y_2\longrightarrow w_{y_2}^j; x_1\longrightarrow y_1; 
			x_2\longrightarrow y_2\longrightarrow z_{y_2}^j; w_{y_2}^j\longrightarrow y_2. 
		\end{displaymath} 
		
		Let $\ML_4^*:=\{Q^*_j:  j\in[m_{y_2}] \}$. Let $\QQ^*_j\in \ML_4^*$, and let $A^*_j$ be an inner vertex of $\QQ^*_j$, then 
		\begin{property}
			\item[\enspace(D4*)] Either $A_j^*\cap XY=XY\setminus \{z_{y_2}^j\}$ or $A^*_j\cap XY=XY$, and either $A_j^*\cap \overline{XY}=\emptyset$
			or $A_j^*\cap \overline{XY}=\{w_{y_2}^j\}$, and at least one of the following holds: $A_j^*\cap XY=XY\setminus \{z_{y_2}^j\}$
			or $A_j^*\cap \overline{XY}=\{w_{y_2}^j\}$.  
		\end{property}
	\end{enumerate}

	Let $\ML:=\ML_1\cup \ML_2\cup \ML_3\cup \ML_4\cup \ML^*_3\cup \ML^*_4$. Since
	$|\ML_1|=2, |\ML_2|=\eta, |\ML_3|=m_{x_1}, |\ML_4|=m_{y_1}, |\ML^*_3|=m_{x_2}, |\ML^*_4|=m_{y_2}$, and
	$m=2+\eta+m_{x_1}+m_{y_1}+m_{x_2}+m_{y_2}$, then in order to finish the \textsc{Step~1} of \textsc{Case~2}, 
	it is enough to show that the paths in $\ML$ are pairwise internally disjoint. 

	\begin{claim}\label{cl:first_constructions2}
		The $X-Y$ paths in $\ML$ are pairwise internally disjoint. 
	\end{claim}
\begin{prf}{\ref{cl:first_constructions2}}
We start by noting that, in some sense, the four ways in which the paths of $\mathbb{T}$ 
	were constructed in \textsc{Step 1} of \textsc{Case 1} have been ``repeated" in the construction of the paths of $\ML$. 
	This close relationship between $\mathbb{T}$ and  $\ML$ is the main ingredient in the proof of Claim~\ref{cl:first_constructions2}.
	
	Before moving on any further, let us verify  that the two paths of $\ML_1$ are internally disjoint.  
	Let $A_1$ and $A_2$ be the inner vertices of $\PP_{x_1}$ and $\PP_{x_2}$, respectively. Then $A_1=(X\setminus\{x_1\})\cup \{y_1\}$ 
	and $A_2=(X\setminus\{x_2\})\cup \{y_2\}$, and  so $A_1\neq A_2$.

	The analogies between the paths of $\mathbb{T}$ and  $\ML$ are given by the interactions that the inner vertices of the $X-Y$ paths
	have with $XY$ and $\overline{XY}'$ in the \textsc{Case 1} and with $XY$ and $\overline{XY}$ in the \textsc{Case 2}. 
	More formally, let ${\mathcal T}\in \MT$ and ${\mathcal L}\in \ML$. 
	We say that ${\mathcal T}$ and ${\mathcal L}$ are \emph{analogous}, if $A\cap \overline{XY}' = B\cap \overline{XY}$ and 
	$A\cap XY  = B \cap XY$, for any $A$ and $B$ inner vertices of ${\mathcal T}$ and ${\mathcal L}$, respectively.
	For ${\mathbb T}'\subseteq {\mathbb T}$ and ${\mathbb L}'\subseteq {\mathbb L}$ we write 
	${\mathbb T}'\sim {\mathbb L}'$ to mean that any path of ${\mathbb T}'$ is analogous to any path of ${\mathbb L}'$. 
	For instance, note that  
	$\MT_1\sim \ML_1$. Indeed, let $\PP_0\in \MT_1$ and  $\PP_{x_i}\in \ML_1$, and
	let $A_0$ and $A$ be inner vertices of $\PP_0$ and  $\PP_{x_i}$, respectively. 
	From (C1) we know that $A_0\cap XY=XY$, and from (D1) we have that $A \cap XY=XY$. Similarly, from (C1) it follows that 
	$A_0\cap \overline{XY}'= \emptyset$, and from (D1) that $A \cap \overline{XY} = \emptyset$. Analogously, we can verify that:
	
	\begin{itemize}
		\item (C1) and (D1) imply that $\MT_1\sim \ML_1$. For completeness of this list, we include this case here again.
		\item (C2), (C2.1) and (D2) imply that $\MT'_2\sim \ML_2$, where $\MT'_2$ is the subset of paths in $\MT_2$ with $w_j\neq v$. 	
		\item (C3) and (D3) imply that $\MT_3\sim \ML_3$. 
		\item (C3) and (D3*) imply that $\MT_3\sim \ML^*_3$. 		
		\item (C4) and (D4) imply that $\MT_4\sim \ML_4$. 
		\item (C4) and (D4*) imply that $\MT_4\sim \ML^*_4$. 
	\end{itemize}
	
	We recall that the strategy in the proof of Claim~\ref{cl:first_constructions} was the following. Given two 
	inner vertices $A$ and $B$ belonging to distinct paths of $\MT$, we always conclude that $A\neq B$ by showing that 
	at least one of  $A\cap \overline{XY}' \neq B\cap \overline{XY}'$ or $A\cap XY \neq B \cap XY$ holds. From this fact, 
	the definition of $\sim$, and the above list, it is not hard to see that analogous arguments as those used in the 
	proof of  Claim~\ref{cl:first_constructions} imply that the $X-Y$ paths belonging to  $\ML_1\cup \ML_2\cup \ML_3\cup\ML_4$ 
	(resp. $\ML_1\cup \ML_2\cup \ML^*_3\cup\ML^*_4$)  are pairwise internally disjoint.
	Thus, it remains to show that the paths in $\ML_3$ (resp. $\ML_4$) are pairwise internally disjoint from the paths in 
	$\ML^*_3\cup \ML^*_4$. 
	
	Let $A_i, A_j, A^*_s$, and $A^*_t$ be inner vertices of $\PP_i\in \ML_3$, $\QQ_j\in \ML_4$, $\PP^*_s\in \ML^*_3$, 
	and $\QQ^*_t\in \ML^*_4$, respectively. We analyze these cases separately.

\begin{mylist}
	\item[$\{A_i,A^*_s\}$:]  By Observation~\ref{obs:intersections}, either $z_{x_1}^i\neq z_{x_2}^s$ or $w_{x_1}^i\neq w_{x_2}^s$. \\
	Suppose that $z_{x_1}^i\neq z_{x_2}^s$. If $A_i\cap XY=XY\setminus \{z_{x_1}^i\}$ or $A^*_s\cap XY=XY\setminus \{z_{x_2}^s\}$, 
	then $A_i\cap XY\neq A^*_s\cap XY$, as required. Suppose then that 
	$A_i\cap XY=XY=A^*_s\cap XY$. From the definitions of $\PP_i$ and $\PP^*_s$ we know that $A_i=(X\setminus \{x_1\})\cup \{w_{x_1}^i\}$
	and $A^*_s=(X\setminus \{x_2\})\cup \{w_{x_2}^s\}$, and so $A_i\neq A^*_s$. \\
	Suppose now that $w_{x_1}^i\neq w_{x_2}^s$. If $A_i\cap \overline{XY}=\{w_{x_1}^i\}$ or $A^*_s\cap \overline{XY}=\{w_{x_2}^s\}$, 
	then $A_i\cap \overline{XY}\neq A^*_s\cap \overline{XY}$. Suppose  then that
	$A_i\cap \overline{XY}=\emptyset=A^*_s\cap \overline{XY}$. Again,  from the definitions of $\PP_i$ and $\PP^*_s$ 
	we have that $A_i=(Y\setminus \{z_{x_1}^i\})\cup \{x_1\}$
	and $A^*_s=(Y\setminus \{z_{x_2}^s\})\cup \{x_2\}$, and so $A_i\neq A^*_s$.  
	\item[$\{A_i,A^*_t\}$:]  Again, by Observation~\ref{obs:intersections}, we have that either $z_{x_1}^i\neq z_{y_2}^t$ or $w_{x_1}^i\neq w_{y_2}^t$. \\
	Suppose that $z_{x_1}^i\neq z_{y_2}^t$. If $A_i\cap XY=XY\setminus \{z_{x_1}^i\}$ or $A^*_t\cap XY=XY\setminus \{z_{y_2}^t\}$, 
	then (D3) and (D4*) imply $A_i\cap XY\neq A^*_t\cap XY$, as required. Suppose then that $A_i\cap XY=XY=A^*_t\cap XY$. From the definitions of 
	$\PP_i$ and $\QQ^*_t$ it follows that $A_i=(X\setminus \{x_1\})\cup \{w_{x_1}^i\}$ and $A^*_t=(Y\setminus \{y_2\})\cup \{w_{y_2}^t\}$, 
	and so $y_1\in A^*_t\setminus A_i$, which implies that $A_i\neq A^*_t$. \\
	Now suppose  that $w_{x_1}^i\neq w_{y_2}^t$. If $A_i\cap \overline{XY}=\{w_{x_1}^i\}$ or $A^*_t\cap \overline{XY}=\{w_{y_2}^t\}$, then  (D3) and (D4*)
	imply that $A_i\cap \overline{XY}\neq A^*_t\cap \overline{XY}$. Suppose then that 
	$A_i\cap \overline{XY}=\emptyset=A^*_t\cap \overline{XY}$. Again,  from the definitions of
	$\PP_i$ and $\QQ^*_t$ we have that $A_i=(Y\setminus \{z_{x_1}^i\})\cup \{x_1\}$ and $A^*_t=(X\setminus \{z_{y_2}^t\})\cup \{y_2\}$, 
	and so $y_1\in A_i\setminus A^*_t$, which implies that $A_i\neq A^*_t$.  
	\item[$\{A_j,A^*_s\}$:] This case can be handled  in the same manner as case $\{A_i,A^*_t\}$. 
	\item[$\{A_j,A^*_t\}$:] Again, this case can be handled  in the same manner as case $\{A_i,A^*_s\}$.
\end{mylist}
\end{prf}
	
	
	\subsubsection{\textsc{Step 2} of \textsc{Case 2}} 
	We recall that $\deg(X)\leq \deg(Y)$ imply that
	
	\begin{equation}\label{myquote:case_2}
			a_1+a_2+b_1+b_2\leq c_1+c_2+d_1+d_2.
	\end{equation}
	We now proceed to show that $\delta - m \leq 1$.

	\begin{claim}\label{cl:min_deg_case2} Let $\delta, m, m_{x_1}, m_{y_1}, m_{x_2}, m_{y_2},$ and $\eta$ be as above. Then, $\delta -m\leq 1$,
		and moreover, if $\delta-m= 1$ then, without loss of generality, we may assume that one of the following holds:
		\begin{enumerate}[label=(J\arabic*)]
	\item \label{J1} $a_1>c_1$, $a_2>c_2$, $b_1>d_1$, $b_2<d_2$ and exactly one of $\{x_2y_1, y_1y_2\}$ is in $T$, 
	\item \label{J2} $a_1>c_1$, $b_1>d_1$, either $a_2<c_2$ or $b_2<d_2$, and exactly one of $\{x_1x_2,y_1y_2\}$ is in $T$, 
	\item \label{J3} $a_1>c_1$, $a_2\le c_2$, $b_1\le d_1$, $b_2< d_2$ and exactly one of $\{x_1x_2, x_2y_1\}$ is in $T$,
	\item \label{J4} $a_2\le c_2$, $b_1\le d_1$, either $a_1<c_1$ or $b_2<d_2$, and $x_1y_2$ is in $T$. 
\end{enumerate}  
	\end{claim}
\begin{prf}{\ref{cl:min_deg_case2}}
We analyze several cases separately, depending on the order relations between the elements
	of the sets $\{a_i,c_i\}$ and $\{b_i,d_i\}$, for $i,j\in \{1,2\}$.  The possible cases are the following: 
	\begin{center}
		\begin{tabular}{|c | c |} 
			\hline
			(1)\quad  $a_1> c_1$, $a_2> c_2$, $b_1> d_1$ and $b_2> d_2$ &
			(9)\quad  $a_1\leq c_1$, $a_2> c_2$, $b_1>  d_1$ and $b_2>d_2$ \\ \hline 
			(2)\quad  $a_1> c_1$, $a_2> c_2$, $b_1>  d_1$ and $b_2\leq d_2$ & 
			(10)\quad $a_1\leq c_1$, $a_2> c_2$, $b_1>  d_1$ and $b_2\leq d_2$ \\ \hline 
			(3)\quad  $a_1> c_1$, $a_2> c_2$, $b_1\leq  d_1$ and $b_2>d_2$ &
			(11)\quad $a_1\leq c_1$, $a_2> c_2$, $b_1\leq  d_1$ and $b_2>d_2$ \\ \hline 
			(4)\quad  $a_1> c_1$, $a_2> c_2$, $b_1\leq  d_1$ and $b_2\leq d_2$ &
			(12)\quad $a_1\leq c_1$, $a_2> c_2$, $b_1\leq  d_1$ and $b_2\leq d_2$ \\ \hline 
			(5)\quad  $a_1> c_1$, $a_2\leq c_2$, $b_1> d_1$ and $b_2> d_2$ &
			(13)\quad $a_1\leq c_1$, $a_2\leq c_2$, $b_1> d_1$ and $b_2> d_2$  \\ \hline  
			(6)\quad  $a_1> c_1$, $a_2\leq c_2$, $b_1> d_1$ and $b_2\leq d_2$ &
			(14)\quad $a_1\leq c_1$, $a_2\leq c_2$, $b_1> d_1$ and $b_2\leq d_2$ \\ \hline 
			(7)\quad  $a_1> c_1$, $a_2\leq c_2$, $b_1\leq d_1$ and $b_2> d_2$ &
			(15)\quad $a_1\leq c_1$, $a_2\leq c_2$, $b_1\leq d_1$ and $b_2> d_2$ \\ \hline 
			(8)\quad  $a_1> c_1$, $a_2\leq c_2$, $b_1\leq d_1$ and $b_2\leq d_2$ &
			(16)\quad $a_1\leq c_1$, $a_2\leq c_2$, $b_1\leq d_1$ and $b_2\leq  d_2$ \\ \hline 
		\end{tabular}
	\end{center}
	As a first observation, the case (1) is impossible 
	because of Inequality~\ref{myquote:case_2}.
	Let us next show that it is enough to consider only six cases: (2), (4), (6), (7), (8) and (16), because each of the rest of cases is
	similar to one of these cases. 
	
	In the cases (3), (9)--(12) and (15) interchange the labels of the elements in each of the
	following sets: 
	$\{x_1,x_2\}$ and $\{y_1,y_2\}$. These interchanges automatically produce the interchange of
	the values in each of the following sets $\{a_1,a_2\}$, $\{b_1,b_2\}$, $\{c_1,c_2\}$ and $\{d_1,d_2\}$. 
	By performing these relabelings, we can see that: case (3)
	is similar to case (2), case (9) is similar to case (5), case (10) is similar
	to case (7), case (11) is similar to case (6), case (12) is similar to case (8), and case (15) is similar to case (14). 
	Thus, we may restrict our analysis to
	the cases (2), (4)--(8), (13), (14), and (16).       
	
	In the cases (5), (13) and (14) we consider the graph $F_{n-k}(T)$ instead of $F_k(T)$
	with the following relabeling. For $i\in \{1,2\}$, let $x'_i:=y_i$
	and $y'_i:=x_i$. Consider the vertices $X'=\phi(X)=V(T)\setminus X$ and 
	$Y'=\phi(Y)=V(T)\setminus Y$ in $F_{n-k}(T)$. Let $XY':=\overline{XY}$ and $\overline{XY}':=XY$, and define 
	the values $a'_i,b'_i,c'_i$ and $d'_i$
	analogously to $a_i,b_i,c_i$ and $d_i$. Then we have $a'_i=b_i$, $b'_i=a_i$, $c'_i=d_i$ and $d'_i=c_i$, 
	and so  case (5) is similar to case (2), case (13) is similar to case (4), and case (14) is similar to case (8). 
	Then, we may assume that one of cases (2), (4), (6), (7), (8) and (16) holds. 
	
	Our strategy is as follows. In any of the analyzed cases we show that $F_k(G)$ has a vertex $X_1$
	``close to" $X$ whose degree is at most $m+1$. Recall that we need to consider only the cases (2), (4), (6), (7), (8) and (16).

	\begin{itemize}
		\item[(2)]  $a_1> c_1$, $a_2> c_2$, $b_1>  d_1$ and $b_2\leq d_2$. 
		
		Then $a_1>0$ and $a_2> 0$. 
		Moreover, our suppositions and (\ref{myquote:case_2}) 
		imply that $d_2 > b_2$. 
		Let $U:=\overline{XY}\cup\{x_1,x_2,y_1,y_2\}$. From $a_1>0,a_2>0,$ and $d_2>0$ it follows that $x_1, x_2,$ 
		and $y_2$ have degree at least 2 in $T[U]$.	Since $T[U]$ is a forest, then there is a vertex 
		$u\in U\setminus \{x_1,x_2,y_1,y_2\}$ such that $deg_{T[U]}(u)\leq 1$. 
		Let $X_1:=(X\setminus \{x_1,x_2\})\cup \{y_1,u\}$. 
		
		\begin{enumerate}
			\item[(2.1)] If $y_1$ is not adjacent to neither $x_2$ nor $y_2$, then 
			\begin{align*}
				\delta\leq \deg(X_1)&\leq c_1+c_2+d_1+b_2+\eta+1+deg_{T[U]}(u)\\
				&\leq m_{x_1}+m_{x_2}+m_{y_1}+m_{y_2}+\eta+2=m.
			\end{align*}	
			
			\item[(2.2)] If $y_1$ is adjacent to some of $x_2$ or $y_2$, then it is adjacent to exactly one of them, 
			because $T$ has no cycles. Hence, in this case
			\begin{align*}
				\delta \leq \deg(X_1)&\leq c_1+c_2+d_1+b_2+\eta+2+deg_{T[U]}(u)\\
				&=m_{x_1}+m_{x_2}+m_{y_1}+m_{y_2}+\eta+3=m+1,
			\end{align*}
			and so \ref{J1} holds.
		\end{enumerate} 
		
		\item[(4)]  $a_1> c_1$, $a_2> c_2$, $b_1\leq  d_1$ and $b_2\leq d_2$. 
		
		If $d_1>0$ and $d_2>0$, then  $\deg_{T[U]}(y_i)=d_i+1\geq 2$ and $\deg_{T[U]}(x_i)=a_i+1\geq 2$, for 
		$U:=\overline{XY}\cup\{x_1,x_2,y_1,y_2\}$ and $i\in \{1,2\}$. These and the fact that $T[U]$ is a forest 
		imply the existence of two vertices 
		$u_1,u_2\in U\setminus \{x_1,x_2,y_1,y_2\}$ such that $\deg_{T[U]}(u_i)\leq 1$ for $i\in\{1,2\}$.  
		Let $X_1:=(X\setminus \{x_1,x_2\})\cup \{u_1,u_2\}$. Then
		\begin{align*}
			\delta\leq \deg(X_1)&\leq c_1+c_2+b_1+b_2+\eta+\deg_{T[U]}(u_1)+\deg_{T[U]}(u_2)\\
			&\leq m_{x_1}+m_{x_2}+m_{y_1}+m_{y_2}+\eta+2=m.
		\end{align*}
		
		We now suppose $d_1>0$ and $d_2>0$ does not hold. Then $d_1=0$ or $d_2=0$. By symmetry, we may assume that $d_1=0$. Then
		$b_1=0$, and $d_2>0$ by (\ref{myquote:case_2}).  Then for $U:=\overline{XY}\cup\{x_1,x_2,y_1,y_2\}$, we have that 
		$\deg_{T[U]}(x_i)=a_i+1\geq 2$ for $i\in \{1,2\}$ and $\deg_{T[U]}(y_2)=d_2+1\geq 2$. Since $T[U]$ has no cycles, 
		then $y_1$ is adjacent to at most one of $x_2$ or $y_2$. From this fact, $b_1=d_1=0$, and $x_1y_1\in E(T[U])$ 
		it follows that $1\leq \deg_{T[U]}(y_1)\leq 2$.
		Again, these and the fact that $T[U]$ is a forest imply the existence of two distinct vertices 
		$u_1,u_2\in U\setminus \{x_1,x_2,y_2\}$ such that $\deg_{T[U]}(u_i)\leq 1$ for $i\in \{1,2\}$. 
		Let $X_1=(X\setminus\{x_1,x_2\})\cup \{u_1,u_2\}$, then
		\begin{align*}
			\delta \leq \deg(X_1)&\leq c_1+c_2+b_1+b_2+\eta+\deg_{T[U]}(u_1)+\deg_{T[U]}(u_2)\\
			&\leq m_{x_1}+m_{x_2}+m_{y_1}+m_{y_2}+\eta+2=m.
		\end{align*}

		\item[(6)]  $a_1> c_1$, $a_2\leq c_2$, $b_1> d_1$ and $b_2\leq d_2$. 
		
		From (\ref{myquote:case_2}) and these inequalities it follows that 
		at least one of $c_2> a_2$ or $d_2>b_2$ holds.  Let $X_1:=(X\setminus\{x_1\})\cup \{y_1\}$. Since $T$ has no cycles, then it 
		contains at most one of $x_1x_2$ or $y_1y_2$.
		\begin{itemize} 
			\item[(6.1)] Suppose that none of  $x_1x_2$ or $y_1y_2$ is in $T$. Then
			\begin{align*}
				\delta\leq \deg(X_1)&\leq c_1+a_2+d_1+b_2+\eta+2\\
				&=m_{x_1}+m_{x_2}+m_{y_1}+m_{y_2}+\eta+2=m.
			\end{align*}
			\item[(6.2)] Suppose that exactly one of $x_1x_2$ or $y_1y_2$ is in $T$. Then
			\begin{align*}
				\delta\leq \deg(X_1)&\leq c_1+a_2+d_1+b_2+\eta+3\\
				&=m_{x_1}+m_{x_2}+m_{y_1}+m_{y_2}+\eta+3=m+1,
			\end{align*}
			and so \ref{J2} holds.
		\end{itemize} 
		
		\item[(7)]  $a_1> c_1$, $a_2\leq c_2$, $b_1\leq d_1$ and $b_2> d_2$. 
		
		Let $X_1:=(X\setminus\{x_1\})\cup \{y_2\}$. 
		Again, since $T$ has no cycles, then there is at most one edge in $T$ with one
		endvertex in $\{x_1,y_1\}$ and the other endvertex in $\{x_2,y_2\}$. Then
		\begin{align*}
			\delta\leq \deg(X_1)&\leq c_1+a_2+b_1+d_2+\eta+1\\
			&=m_{x_1}+m_{x_2}+m_{y_1}+m_{y_2}+\eta+1<m.
		\end{align*}
		
		\item[(8)]  $a_1> c_1$, $a_2\leq c_2$, $b_1\leq d_1$ and $b_2\leq d_2$. 
		
		As we have mentioned above, $T$ 
		has at most one edge with one end in $\{x_1,y_1\}$ and the other end in $\{x_2,y_2\}$.
		\begin{itemize} 
			\item[(8.1)] Suppose that $d_2\leq b_2+1$. Then $X_1:=(X\setminus \{x_1\})\cup \{y_2\}$ satisfies the following
			\begin{align*}
				\delta \leq \deg(X_1)&\leq c_1+a_2+b_1+d_2+\eta+1\\
				&\leq c_1+a_2+b_1+(b_2+1)+\eta+1 \\
				&\leq m_{x_1}+m_{x_2}+m_{y_1}+m_{y_2}+\eta+2=m.
			\end{align*}	
			\item[(8.2)] Suppose that $d_2\geq b_2+2$. Then $a_1>0$ and $d_2\geq 2$, and hence $x_1$ and $y_2$ have degree 
			at least 2 in $T[U]$, for $U:=\overline{XY}\cup \{x_1,y_1,y_2\}$. Since $T[U]$ is a forest, then there is a 
			vertex $u\in U\setminus \{y_1,x_1,y_2\}$    
			such that $\deg_{T[U]}(u)\leq 1$. Let $X_1:=(X\setminus\{x_1\})\cup  \{u\}$. 
			\begin{itemize}
				\item[(8.2.1)] Suppose that $x_2$ is not adjacent to neither $x_1$ nor $y_1$. Then,
				\begin{align*}
					\delta \leq \deg(X_1)&\leq c_1+a_2+b_1+b_2+\eta+2\\
					&= m_{x_1}+m_{x_2}+m_{y_1}+m_{y_2}+\eta+2=m.
				\end{align*}
				\item[(8.2.2)] Suppose that $x_2$ is adjacent to some of $x_1$ or $y_1$. Since there is at most one edge 
				with one end in $\{x_1,y_1\}$ and the other end in 
				$\{x_2,y_2\}$, then $x_2$ is adjacent to exactly one of $x_1$ or $y_1$. Then,
				\begin{align*}
					\delta\leq \deg(X_1)&\leq c_1+a_2+b_1+b_2+\eta+3\\
					&= m_{x_1}+m_{x_2}+m_{y_1}+m_{y_2}+\eta+3=m+1,
				\end{align*}
				implying that \ref{J3} holds.
			\end{itemize}
		\end{itemize} 
		\item[(16)] $a_1\leq c_1$, $a_2\leq c_2$, $b_1\leq d_1$ and $b_2\leq d_2$. 
		
		Since there is at most one edge with one end in 
		$\{x_1,y_1\}$ and the other end in $\{x_2,y_2\}$, then  $T$ contains at most one of $x_1y_2$ or $x_2y_1$. 
		\begin{itemize} 
			\item[(16.1)] Suppose that neither $x_1y_2$ nor $x_2y_1$ is in $T$. Then,	
			\begin{align*}
				\delta\leq \deg(X)&\leq a_1+a_2+b_1+b_2+\eta+2 \\
				&=m_{x_1}+m_{x_2}+m_{y_1}+m_{y_2}+\eta+2=m.
			\end{align*}		
			\item[(16.2)] Suppose that some of $x_1y_2$ or $x_2y_1$ is in $T$. Then exactly one of $x_1y_2$ or $x_2y_1$ belongs to $T$. 
			By symmetry, we may assume that $x_1$ is adjacent to $y_2$. Let $X_1:=(X\setminus\{x_1\})\cup \{y_2\}$. Then, 
			\begin{align*}
				\delta\leq \deg(X_1)&\leq c_1+a_2+b_1+d_2+\eta +1
			\end{align*}		
			\begin{itemize} 
				\item[(16.2.1)] If $a_1=c_1$ and $b_2=d_2$, then  
				\begin{displaymath}\delta\leq \deg(X_1)\leq m_{x_1}+m_{x_2}+m_{y_1}+m_{y_2}+\eta+1\leq m.
				\end{displaymath}
				
				\item[(16.2.2)]  If $a_1<c_1$ or $b_2<d_2$, then 
				\begin{align*}
					\delta\leq \deg(X)&\leq a_1+a_2+b_1+b_2+\eta +3 \\
					&=m_{x_1}+m_{x_2}+m_{y_1}+m_{y_2}+\eta+3=m+1,
				\end{align*}
				and so \ref{J4} holds.
			\end{itemize}
		\end{itemize} 
		
	\end{itemize}
	\end{prf}
	
	Claim~\ref{cl:min_deg_case2} shows that almost all $X-Y$ paths claimed by Lemma~\ref{lem:horse} are provided by 
	$\mathbb{L}$, when $|XY|=k-2$. We finish the proof of \textsc{Case 2} with the construction of the remaining $\delta-m$ $X-Y$ paths. 
	
	\begin{claim}\label{claim:end-case2}
		If  $|XY|=k-2$, then $F_k(T)$ has at least $\delta$  $X-Y$ pairwise internally disjoint paths.
	\end{claim}
\begin{prf}{\ref{claim:end-case2}}
	Consider the $m$ $X-Y$ paths of $\ML$. Clearly, if $m\geq \delta$, then we are done. Then by Claim~\ref{cl:min_deg_case2} 
	we can assume that $m+1=\delta$, and that one of \ref{J1}, \ref{J2}, \ref{J3} or \ref{J4} holds.
	
	In view of these facts,
	it is enough to exhibit a new $X-Y$ path $\PP^{\ell}\notin \ML$ with $\PP^{\ell}$ internally disjoint from any path in $\ML$.
	We note that in any of these four cases, $T$ has one edge $e$ with an endvertex in $\{x_1,y_1\}$ and the other endvertex in $\{x_2,y_2\}$. 
	Since $T$ has no cycles, then $e$ is the only edge of $T$ with this property. Then 
	$\overline{XY}(x_1), \overline{XY}(x_2), \overline{XY}(y_1),\overline{XY}(y_2), XY(x_1),\- XY(x_2), XY(y_1)$, 
	and $XY(y_2)$ are pairwise disjoint, as otherwise $T$ has a cycle. 
	
	Our strategy is as follows. First we define a set $\MP=\{\PP^1, \PP^2, \PP^3,\PP^4\}$ consisting of four new $X-Y$ paths of $F_k(T)$. 
	Then we show that for each of the four cases mentioned in previous paragraph, there is a path in $\MP$ which is internally 
	disjoint from any path of $\ML$, providing the additional required path. 
	
	\begin{enumerate}[1{.}]
	\item If $a_1>c_1$ and $d_2>b_2$, then we define the $X-Y$ path $\PP^1$ as follows: 
	\begin{displaymath}
		\PP^1:=x_1\rightarrow w_{x_1}^{a_1}; x_2\rightarrow y_2\rightarrow w_{y_2}^{d_2}; w_{x_1}^{a_1}\rightarrow x_1\rightarrow y_1;w_{y_2}^{d_2}\rightarrow y_2.
	\end{displaymath}
	From the definition of  $\PP^1$ it follows that if $A^1$ is an inner vertex of $\PP^1$, then 
	
	\begin{property}
		\item[\enspace(E1)] $A^1\cap XY=XY$, and $A^1\cap \overline{XY}\in \left\{\{w_{x_1}^{a_1}\}, \{w_{y_2}^{d_2}\}, \{w_{x_1}^{a_1},w_{y_2}^{d_2}\}\right\}$. 
	\end{property}		
	
	\item If  $a_1>c_1$ and $c_2>a_2$, then we define the $X-Y$ path $\PP^2$ as follows: 
	\begin{displaymath}
		\PP^2:=x_1\longrightarrow w_{x_1}^{a_1};x_2\longrightarrow y_2;z_{x_2}^{c_2}\longrightarrow x_2; 
		w_{x_1}^{a_1}\longrightarrow x_1\longrightarrow y_1; x_2\longrightarrow z_{x_2}^{c_2}.
	\end{displaymath}
	
	From the definition of  $\PP^2$ it follows that if $A^2$ is an inner vertex of $\PP^2$, then 
	
	\begin{property}
		\item[\enspace(E2)]  Either $A^2\cap XY=XY$ or $A^2\cap XY=XY\setminus \{z_{x_2}^{c_2}\}$, 
		and either $A^2\cap \overline{XY}=\emptyset$
		or $A^2\cap \overline{XY}=\{w_{x_1}^{a_1}\}$, and at least one of the following holds:
		$A^2\cap \overline{XY}=\{w_{x_1}^{a_1}\}$ or $A^2\cap XY=XY\setminus \{z_{x_2}^{c_2}\}$.                
	\end{property}	
	
	\item If $c_1>a_1$ and $x_1y_2\in E(T)$, then we define the $X-Y$ path $\PP^3$ as follows: 
	\begin{displaymath}
		\PP^3:=x_1\longrightarrow y_2; z_{x_1}^{c_1}\longrightarrow x_1\longrightarrow y_1; y_2\longrightarrow x_1;
	x_2\longrightarrow y_2; x_1\longrightarrow z_{x_1}^{c_1}.
	\end{displaymath}
	From the definition of  $\PP^3$ it follows that if $A^3$ is an inner vertex of $\PP^3$, then 
	
	\begin{property}
		\item[\enspace(E3)]  $A^3\cap \overline{XY}=\emptyset$, and  $A^3\cap XY\in \left\{XY, XY\setminus \{z_{x_1}^{c_1}\}\right\}$.                     
	\end{property}

	\item If $d_2>b_2$ and $x_1y_2\in E(T)$, then we define the $X-Y$ path $\PP^4$ as follows: 
	\begin{displaymath}
		\PP^4:=x_1\longrightarrow y_2\longrightarrow w_{y_2}^{d_2}; x_2\longrightarrow y_2\longrightarrow x_1\longrightarrow y_1;
	w_{y_2}^{d_2}\longrightarrow y_2.
	\end{displaymath}
	
	From the definition of  $\PP^4$ it follows that if $A^4$ is an inner vertex of $\PP^4$, then 
	\begin{property}
		\item[\enspace(E4)]  $A^4\cap XY=XY$, and $A^4\cap \overline{XY}\in\left\{\emptyset, \{w_{y_2}^{d_2}\}\right\}$.                 
	\end{property}		
\end{enumerate}
	
	We now proceed to show that for $i\in \{1,2,3,4\}$, the $X-Y$ paths in $\{\PP^i\}\cup \ML$ 
	are internally disjoint. For this, let us assume that 
	$A^1,A^2,A^3, A^4, A, A_{i,j}, A_i, A_j, A^*_s$, and $A^*_t$ are inner vertices of 
	$\PP^1, \PP^2, \PP^3, \PP^4, \PP\in \ML_1$, $\PP_{i,j}\in \ML_2$, 
	$\PP_i\in \ML_3$, $\QQ_j\in \ML_4$, $\PP^*_s\in \ML_3^*$, and $\QQ^*_t\in \ML_4^*$, respectively. 
	
		\begin{mylist}
		\item[$\{\PP^1\}\cup \ML$:] 
		We have $A\cap \overline{XY}=\emptyset$ while $A^1\cap \overline{XY}\neq \emptyset$, so $A^1\neq A$. 
		Also we have $A_{i,j}\cap XY\neq XY$ and $A^1\cap XY=XY$, thus $A^1\neq A_{i,j}$.   
		Let $\overline{XY}_1:=\{w_{x_1}^{a_1},w_{y_2}^{d_2}\}$ and 
		$\overline{XY}_2:=(\overline{XY}(x_1)\cup \overline{XY}(x_2)\cup \overline{XY}(y_1)\cup \overline{XY}(y_2)) \setminus \overline{XY}_1$. 
		Note that $\overline{XY}_1$ and $\overline{XY}_2$ are disjoint. 
		For $A'\in \{A_i,A_j,A_s^*,A_t^*\}$ we may assume that $A'\cap \overline{XY}\neq \emptyset$ (as otherwise we have
		$A'\cap \overline{XY}=\emptyset \neq A^1\cap \overline{XY}$, and so $A^1\neq A'$). Then, 
		$A^1\cap \overline{XY}\subset \overline{XY}_1$ while $A'\cap \overline{XY}\subset \overline{XY}_2$,
		since $\overline{XY}_1\cap \overline{XY}_2=\emptyset$, it follows that $A^1\neq A'$. 
		
		\item[$\{\PP^2\}\cup \ML$:] By (E2) we know that $A^2\cap XY\in \{XY, XY\setminus \{z_{x_2}^{c_2}\}\}$. 
		
		First suppose that $A^2\cap XY=XY$, so $A^2\cap \overline{XY}=\{w_{x_1}^{a_1}\}$. 
		Since $A\cap \overline{XY}=\emptyset$ we have $A^2\neq A$. Also, since $A_{i,j}\cap XY\neq XY$, we have 
		$A^2\neq A_{i,j}$. Let $\overline{XY}_1:=\{w_{x_1}^{a_1}\}$ and 
		$\overline{XY}_2:=(\overline{XY}(x_1)\cup \overline{XY}(x_2)\cup \overline{XY}(y_1)\cup \overline{XY}(y_2)) \setminus \overline{XY}_1$. 
		Note that $\overline{XY}_1$ and $\overline{XY}_2$ are disjoint.   For $A'\in \{A_i,A_j,A_s^*,A_t^*\}$ we may 
		assume that $A'\cap \overline{XY}\neq \emptyset$ (as otherwise we have
		$A'\cap \overline{XY}=\emptyset \neq A^2\cap \overline{XY}$, and so $A^2\neq A'$). Then, as in the previous case, we have 
		$A^2\cap \overline{XY}\subset \overline{XY}_1$ while $A'\cap \overline{XY}\subset \overline{XY}_2$,
		since $\overline{XY}_1\cap \overline{XY}_2=\emptyset$, it follows that $A^2\neq A'$.
		
		Suppose now that $A^2\cap XY=XY\setminus \{z_{x_2}^{c_2}\}$. We have $A\cap XY=XY$, so $A^2\neq A$. 
		Let $XY_1:=\{z_{x_2}^{c_2}\}$ and 
		$XY_2:=(XY(x_1)\cup XY(x_2)\cup XY(y_1)\cup XY(y_2))\setminus XY_1$. For $A'\in \{A_i,A_j,A_s^*,A_t^*\}$,
		if $A'\cap XY=XY$ then $A^2\neq A'$. Suppose now that $A'\cap XY\neq XY$.  
		Then, $A'\cap XY\subset XY_2$, while $A^2\cap XY=XY_1$; and since $XY_1\cap XY_2=\emptyset$
		it follows that $A^2\neq A'$.  Consider now the vertex $A_{i,j}$. 
		Note that $z_{x_2}^{c_2}\neq z_i$ or $w_{x_1}^{a_1}\neq w_j$, as otherwise the subgraph of $T$
		induced by $z_i,w_j,x_1,x_2,y_1,y_2,$ and  $e$ contains a cycle.  If $z_{x_2}^{c_2}\neq z_i$, then (D2) and (E2) imply
		$A_{i,j}\cap XY\neq A^2\cap XY$, as required. On the other hand, if $w_{x_1}^{a_1}\neq w_j$, again (D2) and (E2) imply 
		that $A_{i,j}\cap \overline{XY}\neq A^2\cap \overline{XY}$, and so $A_{i,j}\neq A^2$.
		
		\item[$\{\PP^3\}\cup \ML$:] By (E3) we have $A^3\cap \overline{XY}=\emptyset$ and $A^3\cap XY\in \{XY,XY\setminus \{z_{x_1}^{c_1}\}\}$. 
		
		First suppose that $A^3\cap XY=XY$. Then $A^3=(X\setminus \{x_1\})\cup \{y_2\}$, and so 
		$x_2,y_2\in A^3$. On the other hand, for any $A'\in \{A,A_{i,j},A_i,A_j,A_s^*,A_t^*\}$
		we have that $x_2$ and $y_2$ do not belong to $A'$ simultaneously, which implies that $A^3\neq A'$. 
		
		Suppose now that $A^3\cap XY= XY\setminus \{z_{x_1}^{c_1}\}$. In this case proceed in a similar way to
		the case $\{\PP^2\}\cup \ML$ when $A^2\cap XY=XY\setminus \{z_{x_2}^{c_2}\}$.  
		
		\item[$\{\PP^4\}\cup \ML$:] By (E4) we have $A^4\cap XY=XY$ and $A^4\cap \overline{XY}\in \{\emptyset, \{w_{y_2}^{d_2}\}\}$. 
		As a first observation, $A^4\neq A_{i,j}$ because $A_{i,j}\cap XY\neq XY$. 
		
		Suppose that $A^4\cap \overline{XY}=\emptyset$, then $A^4=(X\setminus\{x_1\})\cup \{y_2\}$, and so $x_2,y_2\in A_4$. 
		Similar to case $\{\PP_3\}\cup \ML$, for $A'\in \{A, A_i, A_j, A_s^*, A_t^*\}$
		we have that $x_2$ and $y_2$ do not belong to $A'$ simultaneously. Thus,  $A^3\neq A'$.
		
		Suppose now that $A^4\cap \overline{XY}=\{w_{y_2}^{d_2}\}$. We have $A^4\neq A$ because $A\cap \overline{XY}=\emptyset$. 
		Let $\overline{XY}_1:=\{w_{y_2}^{d_2}\}$ and $\overline{XY}_2:=(\overline{XY}(x_1)\cup \overline{XY}(x_2)\cup \overline{XY}(y_1)\cup \overline{XY}(y_2))\setminus \overline{XY}_1$. 
		Next, for $A'\in \{A_i, A_j, A_s^*, A_t^*\}$ proceed as in the case $\{\PP^1\}\cup \ML$ to show
		that $A^4\neq A'$. 
	\end{mylist}

	Summarizing: for $i\in \{1,2,3,4\}$, we have shown that if $\PP^i$ exists, then $\ML\cup \{\PP^i\}$ is a set of $\delta=m+1$ 
	pairwise internally disjoint $X-Y$ paths of $F_k(T)$. 
	It remains to show that one of $\PP^1,\PP^2,\PP^3,\PP^4$ exists. We have the following: if \ref{J1} holds, 
	then $\PP^1$ exists; if \ref{J2} holds with $a_2<c_2$ (resp. $b_2<d_2$), then $\PP^2$ (resp. $\PP^1$) exists; 
	if \ref{J3} holds, then $\PP^1$ exists; and	if \ref{J4} holds with $a_1<c_1$ (resp. $b_2<d_2$), then $\PP^3$ (resp. $\PP^4$) exists.  
\end{prf}
	
	Clearly, the proof of Claim~\ref{claim:end-case2} finishes the proof of Lemma~\ref{lem:horse}, which implies Theorem \ref{thm:main}.
\end{proof}

\section{Concluding remarks}

The trees and the complete graphs are two families of graphs which are extremely distinct from the point of view of the connectivity.  
Here we have shown that if $G$ is a tree, then $\kappa(F_k(G))=\lambda(F_k(G))=\delta(F_k(G))$. 
Surprisingly, these same equalities hold for the case of the complete graph. More precisely,  
from 
\cite{leanos2018connectivity} and ~\cite{leanos2019edgeconnectivity} we know that the connectivity and the edge-connectivity of 
$F_k(K_n)$ are equal to $\delta(F_k(K_n))$ the minimum degree of $F_k(K_n)$. However, these equalities do not hold in general. 
For instance, it is not hard to see that for the graph $H$ of Figure~\ref{last-fig}
we have $\kappa(F_2(H))=m-1=\lambda(F_2(H))$ and $\delta(F_2(H))=2(m-2)$. 

\begin{figure}[h]
	\begin{center}
	\includegraphics[width=0.26\textwidth]{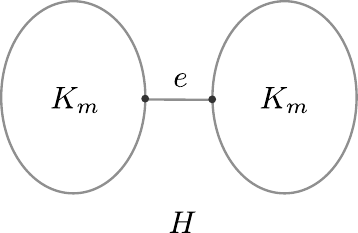}
	\caption{The graph $H$ is constructed by connecting
		two copies of $K_m$ by means of a new edge $e$. }
	\label{last-fig}
	\end{center} 
\end{figure}

On the other hand, based on
computational experimentation and on some analytic approaches we have the following conjecture.  

\begin{conj} 
	If $G$ is a connected graph with girth at least five, then $\kappa(F_k(G))=\delta(F_k(G))$, for each $k\in \{2,3,\ldots , n-2\}$.
\end{conj}

\acknowledgements
\label{sec:acknowl}
This work was initiated at the 3rd Reunion of Optimization, Mathematics, and Algorithms (ROMA 2019), held in Mexico city
in August 2019. We thank the participants for providing a fruitful research environment.
We specially thank Birgit Vogtenhuber and Daniel Perz for various helpful and insightful discussions. 

\nocite{*}
\bibliographystyle{abbrvnat}
\bibliography{conexidad}
\label{sec:biblio}

\end{document}